\documentclass[12pt,reqno]{amsart}
\usepackage[colorlinks=true, pdfstartview=FitV, linkcolor=blue, citecolor=blue, urlcolor=blue]{hyperref}

\usepackage{amssymb,amsmath,amsthm}
\usepackage{amscd}
\usepackage{bbm}
\usepackage{times, verbatim}
\usepackage{graphicx}
\usepackage[english]{babel}
 \usepackage[usenames, dvipsnames]{color}
\usepackage{enumerate}
\usepackage{anysize}
\marginsize{3cm}{3cm}{3cm}{3cm}
\input xy
\xyoption{all}
\usepackage{pb-diagram}
\usepackage[all]{xy}
\input xy
\xyoption{all}

\DeclareFontFamily{OT1}{rsfs}{}
\DeclareFontShape{OT1}{rsfs}{n}{it}{<-> rsfs10}{}
\DeclareMathAlphabet{\mathscr}{OT1}{rsfs}{n}{it}

\DeclareMathOperator{\Ext}{Ext}

\DeclareMathOperator{\Hom}{Hom}
\DeclareMathOperator{\Alg}{Alg}
\DeclareMathOperator{\GL}{GL}

\DeclareMathOperator{\SL}{SL}
\DeclareMathOperator{\G}{G}
\DeclareMathOperator{\M}{M}
\DeclareMathOperator{\V}{V}
\DeclareMathOperator{\St}{St}
\DeclareMathOperator{\Wh}{Wh}

\DeclareMathOperator{\cind}{c-ind}
\DeclareMathOperator{\Ind}{Ind}
\DeclareMathOperator{\csupp}{csupp}
\DeclareMathOperator{\Sym}{Sym}
\DeclareMathOperator{\Irr}{Irr}
\DeclareMathOperator{\lev}{lev}
\DeclareMathOperator{\mat}{Mat}

\newcommand{\btimes}{{\bar{\times}}}
\newcommand{\ropp}{{\bar{r}}}
\newcommand{\csuppline}{{\csupp_{\mathbb{Z}}}}
\newcommand{\R}{{\mathfrak{R}}}
\newcommand{\A}{{\mathcal{A}}}

\newtheorem{theorem}{Theorem}[section]
\newtheorem{proposition}{Proposition}[section]
\newtheorem{corollary}{Corollary}[section]

\theoremstyle{definition}
\newtheorem{remark}{Remark}[section]

\theoremstyle{definition}

\newtheorem{lemma}{Lemma}[section]

\newtheorem{definition}{Definition}[section]

\theoremstyle{definition}
\newtheorem{conjecture}[theorem]{Conjecture}

\theoremstyle{definition}

\numberwithin{equation}{section}

\begin{document}

\title{Non-tempered Ext branching laws for the $p$-adic general linear group}
\author{Mohammed Saad Qadri}
\address{Department of Mathematics \\ Indian Institute of Technology Bombay \\ Mumbai \\ India}
\email{185090012@iitb.ac.in}
\date{}

\maketitle

\begin{abstract}
Let $F$ be a non-archimedean local field. Let $\pi_1$ and $\pi_2$ be irreducible Arthur type representations of $\GL_n(F)$ and $\GL_{n-1}(F)$ respectively. We study Ext branching laws when $\pi_1$ and $\pi_2$ are products of discrete series representations and their Aubert-Zelevinsky duals. We obtain an Ext analogue of the local non-tempered Gan-Gross-Prasad conjecture in this case.
\end{abstract}

\section{Introduction}

Let $F$ be a non-archimedean local field. Following the work \cite{pr18} of Dipendra Prasad, there has been much interest in the study of Ext branching laws about the restriction of an irreducible representation of $\GL_{n}(F)$ to $\GL_{n-1}(F)$. 

Let $\pi_1$ and $\pi_2$ be smooth irreducible representations of $\GL_{n}(F)$ and $\GL_{n-1}(F)$ respectively. It is a fundamental result due to Aizenbud-Gourevitch-Rallis-Schiffman \cite{agrs10} that the space $\Hom_{\GL_{n-1}(F)}(\pi_1,\pi_2)$ is at most one dimensional. It is also known that for all integers $i\geq 1$, the spaces $\Ext^i_{\GL_{n-1}(F)}(\pi_1,\pi_2)$ are finite dimensional (see \cite{as20}, \cite{pr18}). By the study of $\Ext$ branching laws we mean understanding when the spaces $\Ext^i_{\GL_{n-1}(F)}(\pi_1,\pi_2)$ are non-zero. If $\pi_1$ and $\pi_2$ are generic then it is a classical fact (see \cite[Theorem 3]{pr93}, \cite{jpss83}) that $\Hom_{\GL_{n-1}(F)}(\pi_1,\pi_2)=\mathbb{C}$. For the corresponding $\Ext$ branching law, it was conjectured by Dipendra Prasad that when $\pi_1$ and $\pi_2$ are generic then the higher $\Ext$ spaces vanish, that is, $\Ext^i_{\GL_{n-1}(F)}(\pi_1,\pi_2)=0$ for all integers $i\geq 1$. This has been recently proved by K.Y. Chan and G. Savin (see \cite{cs21}). 

The local Gan-Gross-Prasad conjectures \cite{ggp11} were formulated by Wee Teck Gan, Benedict Gross, and Dipendra Prasad giving rules for branching laws of certain pairs $(G,H)$ of classical groups. These conjectures predict when a generic irreducible representation of $H$ occurs in the quotient of a generic irreducible representation of $G$ upon restriction. In a recent paper \cite{ggp20}, they extended these conjectures beyond the generic case to the class of representations of \textit{Arthur type}. For the $p$-adic general linear group, one direction of the local non-generic Gan-Gross-Prasad conjecture was proved by Maxim Gurevich \cite{gu21} and the conjecture was completely settled by K.Y. Chan in the work \cite{ch22}. In a recent work of K.Y. Chan (see \cite{ch23generalbranchinglaw}) the general smooth branching problem for the $p$-adic general linear group is resolved. It is now natural to consider the problem of formulating an $\Ext$ branching law, that is, classifying irreducible representations $\pi_1$ of $\GL_n(F)$ and $\pi_2$ of $\GL_{n-1}(F)$ such that $\Ext^i_{\GL_{n-1}(F)}(\pi_1,\pi_2)\neq 0$ for some integer $i\geq 0$. As pointed out earlier, the $\Ext$ branching for the general linear group is known when both the representations $\pi_1$ and $\pi_2$ are generic, in which case $\Ext$ branching reduces to just $\Hom$ branching.

In Remark 5.8 of the work \cite{ggp20}, the authors mention that it would be an interesting question to give a precise condition predicting exactly when $\Ext^i_{\GL_{n-1}(F)}(\pi_1,\pi_2)$ is non-zero for some $i\geq 0$, when $\pi_1$ and $\pi_2$ are representations of Arthur type. This seems to be a difficult problem in general. In this paper, we study Ext branching laws when $\pi_1$ and $\pi_2$ are products of discrete series representations and the Aubert-Zelevinsky duals of discrete series representations, as a first step towards answering this question. We obtain an Ext analogue of the local non-tempered Gan-Gross-Prasad conjecture in this case. Note that when either $\pi_1$ or $\pi_2$ is generic then some results on $\Ext$ branching law have been obtained in the work \cite{ch22} of K.Y. Chan and this paper builds on the work \cite{ch22}.

We introduce the notion of \textit{strong} $\Ext$ relevance for pairs of Arthur parameters. In order to motivate our definition we recall the original notion of relevance for Arthur type representations of $\GL_n(F)$. We first recall the definition of an Arthur parameter for $\GL_n(F)$ introduced by Gan-Gross-Prasad in \cite{ggp20}. Let $W_F$ denote the Weil group of the field $F$. An Arthur parameter of $\GL_n(F)$ is an admissible homomorphism, \[ \psi: W_F \times \SL_2(\mathbb{C}) \times \SL_2(\mathbb{C}) \rightarrow \GL_n(\mathbb{C}) \] such that the restriction of $\psi$ to $W_F$ has bounded image and the restriction to the two $\SL_2(\mathbb{C})$ factors are algebraic. The first $\SL_2(\mathbb{C})$ is called the Deligne $\SL_2$ and the second $\SL_2(\mathbb{C})$ is called the Arthur $\SL_2$. We can obtain an L-parameter from an Arthur parameter by sending the Weil-Deligne group $W_F \times \SL_2(\mathbb{C})$ to $W_F \times \SL_2(\mathbb{C}) \times \SL_2(\mathbb{C})$ via the map, \[ \alpha: (w,g) \rightarrow \left(w,g,\begin{pmatrix}
    |w|^{1/2} & 0 \\
    0 & |w|^{-1/2}
\end{pmatrix}\right) \] where, $w\in W_F$ and $g\in \SL_2(\mathbb{C})$.
Given an Arthur parameter $\psi$, we obtain the corresponding L-parameter $\psi\circ \alpha$. Thus by the local Langlands correspondence, we can associate an irreducible representation of $\GL_n(F)$ to any Arthur parameter, which turns out to be unitary. An irreducible representation of $\GL_n(F)$ obtained in this manner is called a representation of Arthur type. 

Let $m, n \in \mathbb{Z}_{\geq 0}$ and let $\pi_1$ and $\pi_2$ be Arthur type representations of $\GL_m(F)$ and $\GL_n(F)$ respectively. Let $\A(\pi_1)$ and $\A(\pi_2)$ denote their respective Arthur parameters. Let $V_d$ denote the unique irreducible representation of $\SL_2(\mathbb{C})$ of dimension $d$, that is, $V_d = \Sym^{d-1}(\mathbb{C}^2)$; we define $V_0 = \{0\}$. In the work \cite{ggp20}, the authors define a pair ($\pi_1$,$\pi_2$) to be relevant if there exist admissible homomorphisms $\{\phi_i\}_{i=1}^{i=r+s}$ of $W_F$ with bounded image and positive integers $a_1,a_2,\cdots a_r, b_{r+1},b_2,$ $\cdots b_{r+s},$ $ c_1,c_2,\cdots c_{r+s}$ such that,
\[ \A(\pi_1) = \sum_{i=1}^{r} \phi_i \otimes V_{c_i} \otimes V_{a_i} \oplus \sum_{i=r+1}^{r+s} \phi_i \otimes V_{c_i} \otimes V_{b_i - 1} \] and, \[ \A(\pi_2) = \sum_{i=1}^{r} \phi_i \otimes V_{c_i} \otimes V_{a_i-1} \oplus \sum_{i=r+1}^{r+s} \phi_i \otimes V_{c_i} \otimes V_{b_i}. \]

It is now known by the work of K.Y. Chan (see \cite[Theorem 1.3]{ch22}) that the above relevance condition controls the branching law for representations of Arthur type.

\begin{theorem} \cite[Theorem 1.3]{ch22}
Suppose that $\pi_1$ and $\pi_2$ are irreducible Arthur type representations of $\GL_n(F)$ and $\GL_{n-1}(F)$ respectively. Then $\Hom_{\GL_{n-1}(F)}(\pi_1,\pi_2)\neq 0$ if and only if the Arthur parameters of the representations $\pi_1$ and $\pi_2$ are relevant.    
\end{theorem}

It is known that the above relevance condition does not control the $\Ext$ branching law for Arthur type representations. This means that there exist Arthur type irreducible representations $\pi_1$ and $\pi_2$ of $\GL_n(F)$ and $\GL_{n-1}(F)$ respectively such that $\pi_1$ and $\pi_2$ are not relevant but $\Ext^i_{\GL_{n-1}(F)}(\pi_1,\pi_2)\neq 0$ for some $i\geq 0$. 

For instance, let $\pi_1= \mathbbm{1}_3$ be the trivial representation of $\GL_3(F)$ and let $\pi_2=\St_2$ be the Steinberg of $\GL_2(F)$. Then we have that \[ \Ext^1_{\GL_{2}(F)}(\mathbbm{1}_3,\St_2) =  \mathbb{C}, \] although in this case, $\mathbbm{1}_3$ and $\St_2$ do not have relevant Arthur parameters. Note that the restriction of the trivial representation of $\GL_3(F)$ to $\GL_2(F)$ is the trivial representation of $\GL_2(F)$.

We now introduce the notion of \textit{strong} $\Ext$ relevance which controls the $\Ext$ branching law in the cases we study in this paper. 

\begin{definition} \label{definition of strong ext relevance}
 Let $m, n \in \mathbb{Z}_{\geq 0}$ and let $\pi_1$ and $\pi_2$ be Arthur type representations of $\GL_m(F)$ and $\GL_n(F)$ respectively. Let $\A(\pi_1)$ and $\A(\pi_2)$ denote their respective Arthur parameters. We say that $\pi_1$ and $\pi_2$ are \textit{strong} $\Ext$ relevant if there exist admissible homomorphisms $\{\phi_i\}_{i=1}^{i=r+s}$ and $\{\psi_i\}_{i=1}^{i=t+u}$ of $W_F$ with bounded image and positive integers $a_1,a_2,\cdots a_r,$ $ b_{r+1},b_{r+2},\cdots b_{r+s},$ $ c_1,c_2,\cdots c_{r+s},$ $d_1,d_2,\cdots d_t, e_{t+1},e_{t+2},$ $\cdots e_{t+u},$ $f_1,f_2,\cdots f_{t+u}$ such that, \[ \A(\pi_1) = \sum_{i=1}^{r} \phi_i \otimes V_{c_i} \otimes V_{a_i} \oplus \sum_{i=r+1}^{r+s} \phi_i \otimes V_{c_i} \otimes V_{b_i - 1} \oplus \sum_{i=1}^{t} \psi_i \otimes V_{f_i} \otimes V_{d_i}  \oplus  \sum_{i=t+1}^{t+u} \psi_i \otimes V_{e_i-1} \otimes V_{f_i} \] and, \[ \A(\pi_2) = \sum_{i=1}^{r} \phi_i \otimes V_{c_i} \otimes V_{a_i-1}  \oplus\sum_{i=r+1}^{r+s} \phi_i \otimes V_{c_i} \otimes V_{b_i} \oplus \sum_{i=1}^{t} \psi_i \otimes V_{d_i-1} \otimes V_{f_i} \oplus  \sum_{i=t+1}^{t+u} \psi_i \otimes V_{f_i} \otimes V_{e_i}. \]   
\end{definition}

\begin{remark}
The definition stated above makes sense for arbitrary integers $m$ and $n$. Later we shall specialise to the case $m=n+1$.     
\end{remark}

\begin{remark} 
A priori there are eight possible terms that might occur in the definition of strong $\Ext$ relevance where $\Ext$ might be non-zero based on the GGP relevance together with application of the Aubert-Zelevinsky involution. We enumerate these possibilities below where the first term is for $\GL_m(F)$ and the second term is for $\GL_n(F)$.
\begin{enumerate}
    \item $\phi \otimes V_a \otimes V_b$ and $\phi \otimes V_a \otimes V_{b-1}$. 
    \item $\phi \otimes V_a \otimes V_{b-1}$ and $\phi \otimes V_a \otimes V_{b}$.
    \item $\phi \otimes V_a \otimes V_{b}$ and $\phi \otimes V_{b-1} \otimes V_{a}$.
    \item $\phi \otimes V_{b-1} \otimes V_{a}$ and $\phi \otimes V_a \otimes V_{b}$.
    \item $\phi \otimes V_{a} \otimes V_{b}$ and $\phi \otimes V_b \otimes V_{a-1}$.
    \item $\phi \otimes V_{b} \otimes V_{a-1}$ and $\phi \otimes V_a \otimes V_{b}$.
    \item $\phi \otimes V_a \otimes V_b$ and $\phi \otimes V_{a-1} \otimes V_{b}$. 
    \item $\phi \otimes V_{a-1} \otimes V_{b}$ and $\phi \otimes V_a \otimes V_{b}$.
\end{enumerate} 

But only the sum of the first four is expected to give us a sufficient condition to obtain non-zero $\Ext$ modules. An example which suggests us to rule out some of the eight possibilities is the following. We know from the work of K.Y Chan and G. Savin \cite{cs21} that the Steinberg $\St_n$ of $\GL_{n}(F)$ is projective upon restriction to $\GL_{n-1}(F)$ (and in fact isomorphic to the Gelfand-Graev representation upon restriction to $\GL_{n-1}(F)$) so that, \[ \Ext^i_{\GL_{n-1}(F)}(\St_n,\mathbbm{1}_{n-1}) = 0 \] for all integers $i\geq 1$. Here, $\mathbbm{1}_{n-1}$ denotes the trivial representation of $\GL_{n-1}(F)$. Now by the Euler-Poincare pairing formula of Dipendra Prasad (see \cite{pr18}) we have that, \[ \dim \Wh(\St_n).\dim \Wh(\mathbbm{1}_{n-1}) = \sum_{i\geq 0} (-1)^
{i} \dim \Ext^{i}_{\G_{n-1}}(\St_n,\mathbbm{1}_{n-1}), \] where $\Wh(\St_n)$ and $\Wh(\mathbbm{1}_{n-1})$ are the spaces of Whittaker models for $\St_n$ and $\mathbbm{1}_{n-1}$ respectively. Hence for $n\geq 3$, using the fact that $\Ext^i_{\GL_{n-1}(F)}(\St_n,\mathbbm{1}_{n-1}) = 0$ for all integers $i\geq 1$, along with the Euler-Poincare pairing implies that $\Hom_{\G_{n-1}}(\St_n,\mathbbm{1}_{n-1})=0$. Hence when $n\geq 3$ we have that, \[ \Ext^*_{\GL_{n-1}(F)}(\St_n,\mathbbm{1}_{n-1}) = 0. \] Notice that for $n\geq 3$ the Arthur parameters of $\St_n$ and $\mathbbm{1}_{n-1}$, \[ \A(\St_n) = \mathbbm{1} \otimes V_n \otimes V_1, \]  \[ \A(\mathbbm{1}_{n-1}) = \mathbbm{1} \otimes V_1 \otimes V_{n-1}, \] are not strong $\Ext$ relevant. \textit{Although we do not give an explicit example here, even these eight possibilities do not give a necessary condition for non-vanishing of Ext groups}.

\end{remark}

We prove that strong $\Ext$ relevance gives a necessary and sufficient condition for predicting $\Ext$ non-vanishing when $\pi_1$ and $\pi_2$ are products of discrete series representations and the Aubert-Zelevinsky duals of discrete series representations. Let us state what this condition on $\pi_1$ and $\pi_2$ means, in terms of their Arthur parameters. In this case, the Arthur parameters of $\pi_1$ and $\pi_2$ are a sum of irreducible representations such that for each irreducible direct summand either the Deligne $\SL_2$ or the Arthur $\SL_2$ acts trivially. This means that the Arthur parameters of $\pi_1$ and $\pi_2$ are a direct sum of irreducible representations of the form $\phi \otimes V_a \otimes V_b$ such that either $a=1$ or $b=1$. Here, $\phi$ denotes an irreducible representation of $W_F$ and $V_d$ denotes the unique $d$-dimensional irreducible representation of $\SL_2(\mathbb{C})$. The main result of this paper is the following $\Ext$ branching law. 

\begin{theorem} \label{main theorem for segment type representations}
Let $\pi_1$ and $\pi_2$ be Arthur type representations of $\GL_n(F)$ and $\GL_{n-1}(F)$ respectively. Suppose that the Arthur parameters of $\pi_1$ and $\pi_2$ are of the form, \[ \A(\pi_1) = \sum_{i=1}^{r} \phi_i \otimes V_{a_i} \otimes V_{1} \oplus \sum_{i=r+1}^{r+s} \phi_i \otimes V_1 \otimes V_{b_i} \] and, \[ \A(\pi_2) =  \sum_{i=1}^{t} \psi_i \otimes V_{c_i} \otimes V_{1} \oplus \sum_{i=t+1}^{t+u} \psi_i \otimes V_{1} \otimes V_{d_i}. \] Then, 
\[ \Ext^i_{\GL_{n-1}(F)}(\pi_1,\pi_2)\neq 0 \] for some integer $i\geq 0$, if and only if $\pi_1$ and $\pi_2$ are strong $\Ext$ relevant. 

\end{theorem}

The above result can be thought as an $\Ext$ analogue of the local non-tempered Gan-Gross-Prasad conjecture in this case. 

\begin{remark}
The hypothesis on Arthur parameters in the above theorem amounts to saying that $\pi_1$ and $\pi_2$ are products of unitary representations of segment type (see Section \ref{subsection on definition of speh representations of segment type}). This means that $\pi_1$ and $\pi_2$ are products of unitary representations of the form $Z(\Delta)$ and $Q(\Delta)$ (see Section \ref{section two} for notation).
\end{remark}

We now restate the definition of strong $\Ext$ relevance in representation theoretic terms. We refer to Sections \ref{section two} and \ref{section three} for unexplained notations. Given an irreducible representation $\pi$ of $\GL_n(F)$, we let $\pi^{(h)}$ denote its highest derivative. For the sake of notational convenience, we set $\pi^{-}= \nu^{1/2}\pi^{(h)}$. We let $D(\pi)$ denote the Aubert-Zelevinsky involution of $\pi$.

Let $u_{\rho}(a,b)$ be a Speh representation (see Section \ref{subsection on definition of speh and arthur type representations} for notation) whose corresponding Arthur parameter is given as \[ \phi \otimes V_a \otimes V_b. \] Here $\phi$ is an irreducible representation of $W_F$ corresponding to the cuspidal representation $\rho$. Then by Theorem \ref{Aubert-Zelevinsky Dual of a Speh Representation} (due to Tadic), the Aubert-Zelevinsky dual of $u_{\rho}(a,b)$ is equal to $u_{\rho}(b,a)$ and the corresponding Arthur parameter is equal to \[ \phi \otimes V_b \otimes V_a. \]  Also, by Lemma \ref{derivative of speh representations}, the highest derivative of $\nu^{1/2}u_{\rho}(a,b)$ is equal to $u_{\rho}(a,b-1)$, that is, \[ u_{\rho}(a,b)^- = u_{\rho}(a,b-1). \] This leads us to the following equivalent definition of strong $\Ext$ relevance in representation theoretic terms.

\begin{definition} 
 Let $m, n \in \mathbb{Z}_{\geq 0}$ and let $\pi_1$ and $\pi_2$ be Arthur type representations of $\GL_m(F)$ and $\GL_n(F)$ respectively. Then $\pi_1$ and $\pi_2$ are strong $\Ext$ relevant if there exist Speh representations, 
 \[ \pi_{m,1},\pi_{m,2},\cdots,\pi_{m,r},\pi_{n,1},\pi_{n,2},\cdots,\pi_{n,s}\] and, \[ \pi_{p,1},\pi_{p,2},\cdots,\pi_{p,t},\pi_{q,1},\pi_{q,2},\cdots,\pi_{q,u}\] such that, 
 \[ \pi_1 = \pi_{m,1}\times \cdots \pi_{m,r} \times \pi_{p,1}^- \times \cdots\times \pi_{p,t}^- \times \pi_{n,1}\times \cdots\times \pi_{n,s}  \times  D(\pi_{q,1}^-)\times \cdots\times D(\pi_{q,u}^-) \] and, \[ \pi_2 = \pi_{m,1}^-\times \cdots \pi_{m,r}^- \times \pi_{p,1} \times \cdots\times \pi_{p,t}  \times  D(\pi_{n,1}^-)\times \cdots\times D(\pi_{n,s}^-) \times \pi_{q,1} \times \cdots\times \pi_{q,u}. \]
\end{definition}

Let us now say a few words about what is expected in the general situation when $\pi_1$ and $\pi_2$ are representations of Arthur type. In general, the above notion of strong $\Ext$ relevance is expected to give a sufficient (but not necessary) condition in order to have non-trivial $\Ext$ branching, which we formulate in the following conjecture.

\begin{conjecture} \label{main conjecture for ext non vanishing}
 Let $\pi_1$ and $\pi_2$ be Arthur type representations of $\GL_n(F)$ and $\GL_{n-1}(F)$ respectively. If $\pi_1$ and $\pi_2$ are strong $\Ext$ relevant then, 
\[ \Ext^i_{\GL_{n-1}(F)}(\pi_1,\pi_2)\neq 0 \] for some integer $i\geq 0$.   
\end{conjecture} 

We point out that in general, strong $\Ext$ relevance does not give us a necessary condition for $\Ext$ non-vanishing. This means that there exist Arthur type irreducible representations $\pi_1$ and $\pi_2$ of $\GL_n(F)$ and $\GL_{n-1}(F)$ respectively such that $\Ext^i_{\GL_{n-1}(F)}(\pi_1,\pi_2)\neq 0$ for some $i\geq 0$, but $\pi_1$ and $\pi_2$ are not strong $\Ext$ relevant. 

Consider the following example. Let $\rho$ denote the trivial representation of $\GL_1$. Consider the representations $\pi_1=u_{\rho}(2,3)$ and $\pi_2=u_{\rho}(3,1)\times \rho \times \rho$ of $\GL_6(F)$ and $\GL_5(F)$ respectively. Then by Theorem 7.4 of \cite{ch22} it follows that, \[ \Ext^*_{\GL_5(F)}(\pi_1,\pi_2) \neq 0. \] Note that in this example $\pi_1$ and $\pi_2$ are not strong $\Ext$ relevant.

\subsection{Remarks on the Proof} When dealing with the branching problem for the $p$-adic general linear group it is natural to invoke the Bernstein-Zelevinsky filtration. A very pleasant situation that might occur is the following. Suppose that one is able to show no piece other than exactly one in the Bernstein-Zelevinsky filtration contributes to the branching law. Moreover, suppose that we are also able to show that higher $\Ext$ modules for the other pieces vanish. Then we are left with exactly one piece to analyse which is quite easy to handle. Unfortunately this pleasant situation does not always occur. 

Therefore rather than using the Bernstein-Zelevinsky filtration we use a variant of it introduced by K.Y. Chan in the work \cite{ch22}. This variant (see Theorem \ref{filtration for parabolically induced modules}) clumps together some pieces of the Bernstein-Zelevinsky filtration and hence is coarser than the Bernstein-Zelevinsky filtration. Due to some considerations involving cuspidal support one can show that no piece except one in this coarser filtration can contribute to the $\Ext$ branching.  Thus we are left with an easier situation to analyse. As a consequence one is able to replace some factors in the given representation by a suitable cuspidal representation. In the course of the proof we use the transfer lemma (Lemma \ref{transfer lemma} due to K.Y. Chan) that allows us to transfer the problem from the pair $(\GL_n,\GL_{n-1})$ to the pair $(\GL_{n+1},\GL_{n})$. The variant of the Bernstein-Zelevinsky filtration introduced by K.Y. Chan and other ideas from \cite{ch22} are key to our proofs.

\subsection{Organization of the Paper}
The organization of the paper is as follows. We introduce some preliminary notions and set up notations in Section \ref{section two}. We then introduce various classes of representations in Section \ref{section three}. In Section \ref{section four}, we recall the coarse variant of the Bernstein-Zelevinsky filtration due to K.Y. Chan and collect some important lemmas from \cite{ch22}. We prove some lemmas about extensions on the same group in Section \ref{section five}. We prove a preparatory reduction lemma in Section \ref{section six}. We prove Theorem \ref{main theorem for segment type representations} in Sections \ref{section 7 on proof of one direction} and \ref{section 8 on proof of other direction}. In Section \ref{section nine} we make a remark about the non-uniqueness of matching of factors in Definition \ref{definition of strong ext relevance} of strong Ext relevance.

\section{Preliminaries} \label{section two}

\subsection{Basic Notations}

Let $F$ be a non-archimedean local field. For a $p$-adic group $G$, let $\Alg(G)$ denote the category of smooth representations of $G$. In this paper we are interested in the study of smooth representations of the $p$-adic general linear group and all representations considered in this paper are smooth. For $n\in \mathbb{Z}_{\geq 0}$, we let $\G_n = \GL_n(F)$. Let $\Irr(\G_n)$ denote the collection of irreducible representations of $\G_n$ and set $\Irr = \cup_{n\geq 0} \Irr(\G_n)$. For $\pi \in \Alg(\G_n)$ we set $n(\pi)=n$ and let $\pi^{\vee}$ denote the smooth dual of $\pi$. 

Given a $p$-adic group $G$ and a closed subgroup $H$ and a representation $\rho$ of $H$, we let $\Ind_H^G(\rho)$ (resp. $\cind_H^G(\rho)$) denote the representation of $G$ obtained by normalized induction (resp. normalized compact induction).

We let $\nu_n$ denote the character of $\G_n$ defined as $\nu_n(g) = |\det(g)|_F$, where $g \in \G_n$. We shall usually suppress the subscript and simply denote the character as $\nu$. 

Given an irreducible representation $\pi$, we let $\chi_{\pi}$ denote its central character. Suppose that $|\chi_{\pi}|=\nu^s$. We set $e(\pi)=s$. 

Let $\R_n$ denote the Grothendieck group of the category of smooth representations of finite length of $\G_n$. Set $\R = \oplus_{n\geq 0} \R_n$. The group $\R$ has the structure of a graded bialgebra where the product (denoted by $\times$) is given by normalized parabolic induction and the co-product (denoted by $\partial$) is defined using the Jacquet functor. 

Suppose $\pi_1\in \Alg(\G_{n_1})$ and $\pi_2\in \Alg(\G_{n_2})$ for some $n_1,n_2 \in \mathbb{Z}_{\geq 0}$. Then $\pi_1\times \pi_2$ denotes the representation of $\G_{n_1+n_2}$ obtained by parabolic induction and is defined as, \[ \pi_1\times \pi_2 = \Ind_{P_{n_1,n_2}}^{\G_{n_1+n_2}} (\pi_1 \otimes \pi_2).\] Here $P_{n_1,n_2}$ denotes the standard parabolic in $\G_{n_1+n_2}$ with Levi $\G_{n_1}\times \G_{n_2}$ corresponding to the partition $(n_1,n_2)$. We consider $\pi_1 \otimes \pi_2$ as a representation of $P_{n_1,n_2}$ by extending it trivially across the unipotent subgroup.

Given an irreducible representation $\pi$ of $\G_n$ there exists a unique multiset $\{\rho_1,\rho_2,\cdots,\rho_r\}$ consisting of cuspidal representations such that $\pi$ is a subquotient of $\rho_1\times \rho_2\times \cdots\times \rho_r$. We call this multiset the cuspidal support of $\pi$ and denote it as $\csupp(\pi)$. We define $\csupp_{\mathbb{Z}}(\pi)$, a set of cuspidal representations of $\G_n$ as, \[ \csupp_{\mathbb{Z}}(\pi)=\{ \nu^m\rho | \rho \in \csupp(\pi), m \in \mathbb{Z} \}\] and we say that elements of $\csupp_{\mathbb{Z}}(\pi)$ lie in the cuspidal lines of $\pi$.

Consider $n_1,n_2,\cdots,n_r, n \in \mathbb{Z}_{\geq 0}$ such that $n_1+n_2+\cdots +n_r=n$. Given $\pi \in \Alg(\G_n)$, we let $r_{(n_1,\cdots,n_r)}(\pi)$ denote the Jacquet module of $\pi$, which is a representation of $\G_{n_1}\times \G_{n_2}\times  \cdots \times \G_{n_r}$. Now, given $\pi \in \Alg(\G_n)$, the co-product $\partial: \R \rightarrow \R \times \R$ is defined as, \[ \partial(\pi) = \sum_{i=0}^{n} r_{(i,n-i)}(\pi). \]

Given $\pi \in \Alg(\G_n)$, we shall denote the opposite Jacquet functor corresponding to the partition $(n_1,n_2)$ of $n$ as $\ropp_{(n_1,n_2)}(\pi)$ which as a consequence of the second adjointness theorem is, \[ \ropp_{(n_1,n_2)}(\pi) \simeq (r_{(n_1,n_2)}(\pi^{\vee}))^{\vee}. \]

\subsection{Zelevinsky and Langlands Classification}

Both the Zelevinsky and Langlands classification parameterize any irreducible representation of $\G_n$ in terms of a collection of segments as explained below.

Consider $a, b \in \mathbb{C}$ such that $b-a\in \mathbb{Z}_{\geq 0}$. Given a cuspidal representation $\rho$ of $\G_n$, a segment $\Delta$ associated to the datum ($\rho$,$a$,$b$) is an ordered set of irreducible representations of $\G_n$ of the form, 
\[ \Delta = [a,b]_{\rho} = \{\nu^a\rho,\nu^{a+1}\rho,\cdots,\nu^b\rho\}. \] 
 We set $a(\Delta)=a$ and let $b(\Delta)=b$. If $\rho$ is the trivial representation of $\G_1$ then we may omit it from our notation and simply write the the segment as $\Delta = [a,b]$.

Given the segment $\Delta=[a,b]_{\rho}$, the associated principal series, \[ \nu^a\rho\times \nu^{a+1}\rho\times \cdots \times \nu^b\rho\] has a unique irreducible submodule and unique irreducible quotient which we denote as $Z(\Delta)$ and $Q(\Delta)$ respectively.

We say that two segments $\Delta_1$ and $\Delta_2$ are linked if $\Delta_1 \not \subset \Delta_2$, $\Delta_2 \not \subset \Delta_1$ and $\Delta_1 \cup \Delta_2$ is a segment. We say that $\Delta_1$ precedes $\Delta_2$ if $\Delta_1$ and $\Delta_2$ are linked and $b(\Delta_2)= \nu^c b(\Delta_1)$ for some $c\in \mathbb{Z}_{>0}$. 

Suppose we are given a multiset of segments (which we call a multisegment) $\mathfrak{m} = \{ \Delta_1,\Delta_2,\cdots,\Delta_r \}$ such that $\Delta_i$ does not precede $\Delta_j$ for any $i<j$. The representation \[ Z(\Delta_1)\times Z(\Delta_2) \times \cdots \times Z(\Delta_r)\] has a unique irreducible submodule which we denote as $Z(\Delta_1,\Delta_2,\cdots,\Delta_r)$. By the Zelevinsky classification any irreducible representation of $\G_n$ can be uniquely realized in this manner. Similarly by the Langlands classification, the representation \[ Q(\Delta_1)\times Q(\Delta_2) \times \cdots \times Q(\Delta_r)\] has a unique irreducible quotient which we denote as $Q(\Delta_1,\Delta_2,\cdots,\Delta_r)$.

\subsection{The Aubert-Zelevinsky Involution} 

Given a irreducible representation $\pi$ of $\GL_n(F)$ the Aubert-Zelevinsky Involution (denoted by the symbol $D$) takes $\pi$ to another irreducible representation of $\GL_n(F)$. This involution can be also defined in the more general context of any reductive $p$-adic group. Since we are only concerned with the general linear group in this paper we only recall the definition for $\GL_n$. 

Recall that $\R_n$ denotes the Grothendieck group of the category of smooth representations of finite length of $\G_n$. The ring $\R_n$ is a polynomial ring in the indeterminates of the form $Z(\Delta)$ (see \cite[Corollary 7.5]{ze80}). Therefore it is enough to define the involution $D$ on elements of the form $Z(\Delta)$. We define, \[ D: \R_n \rightarrow \R_n \] by sending $Z(\Delta)$ to $Q(\Delta)$. It is known that the involution $D$ takes irreducible representations to irreducible representations.

We have the following cohomological duality theorem (see \cite[Theorem 2]{np20}) due to Nori and Prasad.

\begin{theorem} \label{duality theorem of nori and prasad}
Let $\pi_1$ be a smooth irreducible representations of $\G_n$ and $\pi_2$ be any smooth representation of $\G_n$. Then for any integer $i\geq 0$, \[ \Ext^i_{\G_n}(\pi_1,\pi_2)^{\vee} \simeq \Ext^{d(\pi_1)-i}_{\G_n}(\pi_2,D(\pi_1)).\] Here $d(\pi_1)$ is the split rank of the Levi subgroup of $\G_n$ which carries the cuspidal support of $\pi_1$. Here, $^{\vee}$ denotes the dual as a linear vector space.
\end{theorem}

\subsection{On Derivatives and the Bernstein-Zelevinsky Filtration} \label{subsection on derivatives}

Given $\pi\in \Alg(\G_n)$, we want to define the $i$-th Bernstein-Zelevinsky derivative $\pi^{(i)}$, which is a smooth representation of $\G_{n-i}$. We first recall the definitions of the Bernstein-Zelevinsky derivative functors that we need. Let $\M_n\subset \G_n$ denote the mirabolic subgroup inside $\G_n$. We have the decomposition, \[ \M_n = \G_{n-1}\cdot \V_n \] where, $\V_n \cong F^{n-1}$ denotes the unipotent radical of $\M_n$. We define functors, \[ \Phi^- : \Alg(\M_n) \rightarrow \Alg(\M_{n-1}) \] and, \[ \Psi^- : \Alg(\M_n) \rightarrow \Alg(\G_{n-1}) \] as follows.

Given $\tau\in \Alg(\M_n)$ we define, \[  \Phi^-(\tau) = \delta^{-1/2} \tau / \langle \tau(u).x - \psi_n^{\prime}(u).x : x \in \tau, u\in V_n
 \rangle. \] Here, $\delta^{-1/2}$ is the modulus character of $\M_n$ and $\psi^{\prime}_n$ is a non-degenerate character of $\V_n$ defined as, $\psi^{\prime}_n(a_1,a_2,\cdots,a_{n-1})=\psi(a_{n-1})$. Here, $\psi$ is a non-trivial additive character of the field $F$. Note that $\M_{n-1}$ normalizes $\V_n$ and $\psi^{\prime}_n$ and therefore we may consider $\Phi^-(\tau)$ to be a smooth representation of $\M_{n-1}$.

 Similarly we define, \[  \Psi^-(\tau) = \delta^{-1/2} \tau / \langle \tau(u).x - x : x \in \tau, u\in V_n
 \rangle. \] Once again, $\delta^{-1/2}$ denotes the modulus character of $\M_n$. Since $\G_{n-1}$ normalizes $\V_n$ we have that $\Psi^-(\tau)\in \Alg(\G_{n-1})$.

 Now given $\pi\in \Alg(\G_n)$, we define the $i$-th Bernstein-Zelevinsky derivative $\pi^{(i)}\in \Alg(\G_{n-i})$ as, \[ \pi^{(i)} = \Psi^-(\Phi^-)^{i-1}(\pi|_{\M_n}). \]
 
Given, $\pi\in \Alg(\G_n)$, let $k$ be the largest integer such that $\pi^{(k)}\neq 0$. Then we say that the level of $\pi$ is $k$
and write $\lev(\pi)=k$.

We also set the following notation. Let $U_n \subset \G_n$ be the subgroup of upper triangular unipotent matrices. Fix a non-degenerate character $\psi_n$ of $U_n$. Then we denote by $\Pi_n$ the representation of $\M_n$ defined as, \[ \Pi_n = \cind_{U_n}^{\M_n}\psi_n. \]

\subsubsection{Product Rule for Derivatives}

We shall require the following product rule for derivatives (see \cite[Corollary 4.14(c)]{bz77}).
\begin{lemma}
Let $\pi_j \in \Alg(\G_{n_j})$ for $j=1,2,\cdots,k$. Then the representation $(\pi_1 \times \pi_2 \times \cdots \times \pi_k)^{(i)}$ has a filtration whose successive quotients are $\pi_1^{(i_1)} \times \pi_2^{(i_2)} \times \cdots \times \pi_k^{(i_k)}$, where, $0\leq i_t \leq n_t$ for $t=1,2,\cdots,k$ and $i_1+i_2+\cdots+i_k = i$.   
\end{lemma}

\subsubsection{Bernstein-Zelevinsky Filtration}

Let $\pi\in \Alg(\G_n)$. By the Bernstein-Zelevinsky filtration (see \cite{bz76}), we have a filtration of submodules $ 0=V_n\subsetneq V_{n-1}\subsetneq \cdots \subsetneq V_0 = \pi|_{\G_{n-1}}$ such that,
\[ V_i/V_{i+1} \cong \nu^{1/2}\pi^{(i+1)}\times \cind_{U_i}^{\G_i}\psi_i. \]
Here $U_i\subset \G_i$ is the subgroup of upper triangular unipotent matrices and $\psi_i$ is a fixed non-degenerate character of $U_i$. 

\subsection{Restriction of a Principal Series to the Mirabolic Subgroup} \label{subsection on restriction of a principal series to the mirabolic subgroup} In this subsection, we want to recall a key lemma (see \cite[Proposition 4.13 (a)]{bz77}) about the restriction of a principal series representation of $\GL_n$ to the mirabolic subgroup. We first recall the definitions of two types of mirabolic induction that we shall need from \cite[Section 4.12]{bz77}). 

Let $\M_m$ denote the mirabolic subgroup sitting inside $\G_m$. Let $\mat_{m,p}$ denote the set of $m \times p$ matrices. Let $0_{m,p}$ denote the zero matrix of order $m \times p$.

Suppose $r$ and $s$ be integers and let $n=r+s$. Let $\rho \in \Alg(\G_r)$ and $\tau \in \Alg(\M_s)$. We define $\rho \btimes \tau \in \Alg(\M_n)$ and $\tau \btimes \rho \in \Alg(\M_n)$ as follows.

\subsubsection{Definition of $\rho \btimes \tau$} Consider the subgroup of $\M_n$ defined as follows. \[ H = \left\{ \begin{pmatrix}
g & u \\
  & m  \\
\end{pmatrix} : g\in \G_{r},m\in \M_{s}, u\in \mat_{r,n-r} \right \}. \] We extend the representation $\rho \otimes \tau$ of $\G_r \times \M_s$ trivially to $H$. We define, \[ \rho \btimes \tau = \Ind_H^{\M_n} (\rho \otimes \tau). \]

\subsubsection{Definition of $\tau \btimes \rho$} Consider the subgroup of $\M_n$ defined as follows. \[ H^{\prime} = \left\{ \begin{pmatrix}
g^{\prime} & u & m^{\prime} \\
  & g & 0_{r,1} \\ 
  &   & 1 \\  
\end{pmatrix} : g\in \G_{r},g^{\prime}\in \G_{s-1}, m^{\prime} \in \mat_{s-1,1}, u\in \mat_{s-1,r} \right \}. \]

Set $\Tilde{\rho}= \nu^{-1/2}\rho$. We define a representation $\pi$ of $H^{\prime}$ acting on the underlying space of $\tau \otimes \rho$ as, 
\[ \pi \begin{pmatrix}
g^{\prime} & u & m^{\prime} \\
  & g & 0_{r,1} \\ 
  &   & 1 \\  
\end{pmatrix}  = \tau\begin{pmatrix}
g^{\prime} & m^{\prime} \\
0_{1,s-1}  & 1 \\  
\end{pmatrix} \otimes \Tilde{\rho}(g). \] Then $\tau \btimes \rho$ is defined as, \[ \tau \btimes \rho = \Ind_{H^{\prime}}^{\M_n} \pi. \]

We are now in a position to state the required lemma. The lemma below tells us how a principal series representation decomposes upon restriction to the mirabolic subgroup.

\begin{lemma} \label{lemma about restriction to the mirabolic subgroup}\cite[Proposition 4.13(a)]{bz77}
 Let $\pi_1$ and $\pi_2$ be smooth representations of $\G_{n_1}$ and $\G_{n_2}$ respectively. Then \[ 0\rightarrow \pi_1|_{\M_{n_1}}\btimes\pi_2 \rightarrow (\pi_1\times\pi_2)|_{\M_{n_1+n_2}} \rightarrow \pi_1\btimes\pi_2|_{\M_{n_2}} \rightarrow 0. \]     
Here $\M_{n_1}$ and $\M_{n_2}$ are the mirabolic subgroups of $\G_{n_1}$ and $\G_{n_2}$ respectively.
 
\end{lemma}

\subsection{Second Adjointness and the Identity orbit in the Geometric Lemma} \label{section on geometric lemma}

Let $\pi_1\in \Alg(\G_{n_1})$ and $\pi_2\in \Alg(\G_{n_2})$ for some $n_1,n_2 \in \mathbb{Z}_{\geq 0}$. Set $n=n_1+n_2$ and consider the representation $\pi_1\times \pi_2$ of $\G_{n}$. The geometric lemma (see \cite{bz77}) describes the factors in a composition series of the Jacquet module $r_{(n_1,n_2)}(\pi_1\times \pi_2)$. Each factor in the composition series of $r_{(n_1,n_2)}(\pi_1\times \pi_2)$ corresponds to a double coset representative of $P_{n_1,n_2}\setminus \G_{n}/P_{n_1,n_2}$ where, $P_{n_1,n_2}$ denotes the parabolic inside $\G_n$ corresponding to the partition $(n_1,n_2)$. Here, the term $\pi_1\otimes \pi_2$ in the composition series of $r_{(n_1,n_2)}(\pi_1\times \pi_2)$ corresponds to the identity double coset. By the geometric lemma the term $\pi_1\otimes \pi_2$ occurs as the quotient of $r_{(n_1,n_2)}(\pi_1\times \pi_2)$.

We want to understand the structure of the opposite Jaquet module of $\ropp_{(n_1,n_2)}(\pi_1\times \pi_2)$. Since we know that, \[ \ropp_{(n_1,n_2)}(\pi_1\times \pi_2) \simeq (r_{(n_1,n_2)}((\pi_1\times \pi_2)^{\vee})))^{\vee}\] the above discussion leads to the following corollary. 

\begin{corollary}
 Let $\pi_1\in \Alg(\G_{n_1})$ and $\pi_2\in \Alg(\G_{n_2})$ for some $n_1,n_2 \in \mathbb{Z}_{\geq 0}$. Then $\pi_1\otimes \pi_2$ occurs as a submodule of $\ropp_{(n_1,n_2)}(\pi_1\times \pi_2)$.  
\end{corollary}

\subsection{Kunneth Formula and a lemma on Ext Vanishing}

We shall require the following Kunneth formula in the course of our calculations of $\Ext$ modules. 
\begin{theorem}(\cite[Theorem 3.5]{pr24}) \label{kunneth formula}
Let $H_1$ and $H_2$ be $p$-adic groups and suppose that $H_1$ is reductive. Let $E_1$ and $F_1$ (resp. $E_2$ and $F_2$) be smooth representations of $H_1$ (resp. $H_2$). If $E_1$ and $F_1$ have finite lengths then, \[ \Ext^i_{H_1\times H_2} (E_1\otimes E_2,F_1\otimes F_2) = \sum_{k=0}^i \Ext^k_{H_1} (E_1,F_1) \otimes \Ext^{i-k}_{H_2} (E_2, F_2). \]
\end{theorem} The above Kunneth formula when combined with cuspidal support considerations allows us to conclude $\Ext$ vanishing in some cases as in the lemma below. The lemma below is a slight reformulation of \cite[Lemma 2.4]{ch22}.

\begin{lemma} \label{lemma about standard argument on ext vanishing}
Let $\pi_1$ be an irreducible representation of $\G_{n-k}$ and $\pi_2$ be a smooth representation of $\G_{k}$ not necessarily of finite length. Let $\pi$ be an irreducible representation of $\G_n$. Let $\sigma$ be a cuspidal representation such that $\sigma\in \csupp(\pi_1)$ but $\sigma\not \in \csupp(\pi)$. Then for all integers $i\geq 0$,
\[ \Ext^i_{\G_{n}}(\pi_1\times \pi_2,\pi) = 0.\]  
\end{lemma}

\begin{remark}
The above lemma codifies a standard argument that we shall use. In the above lemma, since $\pi_2$ is not necessarily of finite length we cannot directly conclude $\Ext$ vanishing by cuspidal support considerations. Instead we first apply second adjointness and the Kunneth formula for each simple composition factor of the opposite Jacquet module of $\pi$. Now by comparing cuspidal supports at $\sigma\in \csupp(\pi_1)$ and the corresponding factor in the opposite Jacquet module of $\pi$ we conclude $\Ext$ vanishing.
\end{remark}

\subsection{Generic Branching Law}

We recall the $\Hom$ and $\Ext$ branching law for generic representations. This will form the base case for the inductive proof of our theorems.

\begin{theorem}(\cite[Theorem 4.1]{cs21}) \label{generic branching law}
Let $n\geq 0$ be an integer. Let $\pi_1$ and $\pi_2$ be irreducible generic representations of $\G_n$ and $\G_{n-1}$ respectively. Then
\[ \Hom_{\G_{n-1}}(\pi_1,\pi_2) = \mathbb{C},\] and for all integers $i\geq 1$,
\[ \Ext^i_{\G_{n-1}}(\pi_1,\pi_2) = 0.\]
\end{theorem}

The above theorem about higher $\Ext$ vanishing was later extended to the case when $\pi_1$ and the dual of $\pi_2$ are standard representations by K.Y. Chan (see \cite[Theorem 1.1]{ch23}).

\section{On Classes of Representations of $\GL_n(F)$} \label{section three}

In this Section, we introduce various classes of representations of the $p$-adic general linear group. Let $\Irr^c(\G_n)$ (resp. $\Irr^u(\G_n)$) denote the collection of irreducible cuspidal (resp. unitary) representations of $\G_n$. Let $\Irr^{c,u}(\G_n) = \Irr^c(\G_n) \cap \Irr^u(\G_n)$ denote the class of cuspidal unitary representations. We set $\Irr^c = \cup_{n\geq 0} \Irr^c(\G_n)$, $\Irr^u = \cup_{n\geq 0} \Irr^u(\G_n)$ and $\Irr^{c,u} = \cup_{n\geq 0} \Irr^{c,u}(\G_n)$.

\subsection{On Speh and Arthur Representations} \label{subsection on definition of speh and arthur type representations}

The Speh representations are defined as follows. Let $\rho \in \Irr^{c,u}$. For $a\in \mathbb{Z}_{\geq 0}$ let $\Delta(\rho,a)$ be the segment $\Delta(\rho,a)=[{-\frac{(a-1)}{2}},{\frac{(a-1)}{2}}]_{\rho}$. We let $\delta_{\rho}(a)$ denote the square integrable representation $\delta_{\rho}(a)=Q(\Delta(\rho,a))$.

Let $b$ be a non-negative integer. The product, \[ \nu^{\frac{(b-1)}{2}}\delta_{\rho}(a) \times \nu^{\frac{(b-1)}{2}-1}\delta_{\rho}(a) \times \cdots \times \nu^{-\frac{(b-1)}{2}}\delta_{\rho}(a) \] has a unique irreducible quotient which we denote as $u_{\rho}(a,b)$. The irreducible representations of the form $u_{\rho}(a,b)$ are known as Speh representations.
\begin{definition} We say that an irreducible representation $\pi$ of $\G_n$ is of Arthur type if there exist Speh representations $\pi_1,\pi_2,\cdots,\pi_k$ such that, \[ \pi = \pi_1 \times \pi_2 \times \cdots \times \pi_k. \] 
\end{definition}

\subsubsection{Derivatives of Speh Representations} We shall gather some information about the derivatives of Speh representations. The knowledge of cuspidal support of the derivatives of Speh representations will be helpful in allowing us to conclude $\Ext$ vanishing.

The lemma below follows from \cite[Theorem 14]{lm14} when it is reformulated in terms of the Langlands classification. We refer to the discussion after Theorem 14 in loc. sit. Given $\pi\in \Alg(\G_n)$, recall that by the level of $\pi$ (denoted $\lev(\pi)$) we mean the largest integer $k$ such that $\pi^{(k)}\neq 0$.

\begin{lemma} \label{derivative of speh representations}
 Let $\pi = u_{\rho}(a,b)$ be a Speh representation. Then the level of $\pi$ is equal to $n(\rho)a$, that is, $\lev(\pi)=n(\rho)a$. 
 \begin{enumerate}
     \item If $k<\lev(\pi)$ then $\nu^{(a+b-2)/2}\rho \in \csupp(\pi^{(k)})$.
     \item If $k=\lev(\pi)$ then $\pi^{(k)} = \nu^{-1/2}u_{\rho}(a,b-1)$.
 \end{enumerate}
\end{lemma}

\subsubsection{Aubert-Zelevinsky Dual of a Speh Representation}

We shall need the following calculation of the Aubert-Zelevinsky dual of a Speh representation. The following result is due to Tadic (see \cite[Theorem B]{ta86}).
\begin{theorem} \label{Aubert-Zelevinsky Dual of a Speh Representation}
Let $a,b\in \mathbb{Z}_{\geq 0}$ and $\rho\in \Irr^c$. Let $\pi=u_{\rho}(a,b)$ be a Speh representation of $\G_n$. Then, \[ D(u_{\rho}(a,b)) = u_{\rho}(b,a). \]   
\end{theorem}

The above result can be interpreted in the language of Arthur parameters as follows. Let $\pi=u_{\rho}(a,b)$ be a Speh representation of $\G_n$ that arises from an Arthur parameter of the form, \[ \psi \otimes \Sym^a(\mathbb{C}^2) \otimes \Sym^b(\mathbb{C}^2). \] Then the Aubert-Zelevinsky dual of $u_{\rho}(a,b)$ is once again a Speh representation whose corresponding Arthur parameter is given as, \[ \psi \otimes \Sym^b(\mathbb{C}^2) \otimes \Sym^a(\mathbb{C}^2). \]

\subsubsection{Definition of Speh Representations of Segment Type} \label{subsection on definition of speh representations of segment type} We introduce some terminology for the sake of convenience. Suppose that we are given a Speh representation $u_{\rho}(a,b)$ of $\G_n$. If $a=1$ then the Speh representation $u_{\rho}(a,b)$ is of the form $Z(\Delta)$ for the segment $\Delta=[-\frac{(b-1)}{2},\frac{(b-1)}{2}]_\rho$. If $b=1$ then the Speh representation $u_{\rho}(a,b)$ is of the form $Q(\Delta)$ for the segment $\Delta=[-\frac{(a-1)}{2},\frac{(a-1)}{2}]_\rho$. 

If a Speh representation is of the form $Q(\Delta)$ or $Z(\Delta)$ for some segment $\Delta$ then we shall say that the Speh representation is of \textit{segment type}. If $\pi\in \Irr(\G_n)$ is of segment type then its Arthur parameter $A(\pi)$ is such that either the Deligne $\SL_2$ factor of $A(\pi)$ or the Arthur $\SL_2$ factor of $A(\pi)$ acts trivially.

We shall require the following computations of Jacquet modules (see \cite[Proposition 3.4]{ze80} and \cite[Proposition 9.5]{ze80}).

\begin{lemma} \label{lemma about jacquet modules of representations of segment type}
Let $\rho$ be a cuspidal representation of $\G_m$ and consider the segment $\Delta=[\rho,\rho^{k-1}\rho]$ where, $n=km$. If $l$ is not divisible by $m$ then $r_{(n-l,l)}(Q(\Delta))= \ropp_{(n-l,l)}(Q(\Delta))= r_{(n-l,l)}(Z(\Delta))= \ropp_{(n-l,l)}(Z(\Delta))= 0$. If $l=mp$ then, \[ r_{(n-l,l)}(Q(\Delta))= Q([\nu^p\rho,\nu^{k-1}\rho]) \otimes Q([\rho,\nu^{p-1}\rho]). \]  \[ \bar{r}_{(n-l,l)}(Q(\Delta))= Q([\rho,\nu^{k-p-1}\rho]) \otimes Q([\nu^{k-p}\rho,\nu^{k-1}\rho]). \]  \[ r_{(n-l,l)}(Z(\Delta))= Z([\rho,\nu^{k-p-1}\rho]) \otimes Z([\nu^{k-p}\rho,\nu^{k-1}\rho]). \] \[ \ropp_{(n-l,l)}(Z(\Delta))= Z([\nu^p\rho,\nu^{k-1}\rho]) \otimes Z([\rho,\nu^{p-1}\rho]). \]
\end{lemma}

\begin{remark} \label{remark on jacquet modules of representations of segment type}
Suppose $Q(\Delta)$ is a unitary representation of $\G_n$. Suppose $r_{(n-l,l)}(Q(\Delta)) = \omega_1 \otimes \omega_2$ where, $r_{(n-l,l)}(Q(\Delta))$ is the Jacquet module of $Q(\Delta)$ with respect to a proper parabolic of $\G_n$ corresponding to the partition $(n-l,l)$ of $n$. Then by the first computation in the above lemma it is evident that $e(\omega_1)>0$ whereas, $e(\omega_2)<0$. We recall that given an irreducible representation $\pi$, $e(\pi)$ denotes the exponent of the central character of $\pi$. We can make analogous assertions for the remaining three cases of the above lemma.
\end{remark}

\section{A Coarse Variant of the Bernstein-Zelevinsky Filtration} \label{section four}

In this Section, we recall the filtration for parabolically induced representations introduced by K.Y. Chan in the work \cite{ch22}. It will be the main tool in our proofs. We also state two lemmas from \cite{ch22} that follow by an application of this filtration.

\subsection{Filtration for Parabolically Induced Modules} 

Since the following filtration on parabolically induced modules is going to play a key role in proving our results, we recall the main idea behind the proof for the sake of completeness.

\begin{theorem} \cite[Proposition 5.13]{ch22} \label{filtration for parabolically induced modules}
Let $\pi_1$ and $\pi_2$ be smooth representations of $\G_{n_1}$ and $\G_{n_2}$ respectively. Then $(\pi_1\times\pi_2)|_{\G_{n_1+n_2-1}}$ has a filtration,
\[0=V_d\subsetneq V_{d-1}\subsetneq \cdots \subsetneq V_0 = (\pi_1\times\pi_2)|_{\G_{n_1+n_2-1}}\]
such that, $$ V_0/V_1 \cong \nu^{1/2}\pi_1\times\pi_2|_{\G_{n_2-1}}, $$ and for $1\leq k\leq d-1$, 
$$V_k/V_{k+1} \cong \nu^{1/2}\pi_1^{(k)}\times (\Pi_k \btimes \pi_2)|_{\G_{n_2 + k -1}}.$$

Here, $\Pi_k$ denotes the representation of the mirabolic group $\M_k$ as defined in Section \ref{subsection on derivatives} and $\btimes$ denotes the mirabolic induction as defined in Section \ref{subsection on restriction of a principal series to the mirabolic subgroup}.
 
\end{theorem}
\begin{proof}
This is shown in \cite[Proposition 5.13]{ch22} and we recall the main idea of the proof for the sake of completeness. By Lemma \ref{lemma about restriction to the mirabolic subgroup}, we obtain the following exact sequence, \[ 0\rightarrow \pi_1|_{\M_{n_1}}\btimes\pi_2 \rightarrow (\pi_1\times\pi_2)|_{\M_{n_1+n_2}} \rightarrow \pi_1\btimes\pi_2|_{\M_{n_2}} \rightarrow 0. \] On restricting $\pi_1|_{\M_{n_1}}\btimes\pi_2$ to $\G_{n_1+n_2-1}$ we get $\nu^{1/2}\pi_1\times\pi_2|_{\G_{n_2-1}}$ which gives us the first term in our filtration. 

For the submodule $\pi_1|_{\M_{n_1}}\btimes\pi_2$ we apply the Bernstein-Zelevinky filtration to decompose $\pi_1|_{\M_{n_1}}$ and then restrict to $\G_{n_1+n_2-1}$ in order to obtain the rest of the terms of our theorem.  
\end{proof}

\begin{remark}
As pointed out earlier the above filtration can be thought of as a coarse version of the Bernstein-Zelevinsky filtration where we club together some of the pieces in the Bernstein-Zelevinsky filtration. As we shall see in Lemma \ref{lemma for the key reduction step}, an advantage of having this coarser filtration is that we can establish $\Ext$ vanishing for all but one of the pieces of this filtration making our situation easier to analyse. 
\end{remark}

\subsection{Transfer Lemma}

The following lemma shall allow us to transfer a branching problem for the pair $(\G_n, \G_{n-1})$ to a branching problem for the pair $(\G_{n-1}, \G_{n-2})$ and vice-versa. 

\begin{lemma}\cite[Proposition 4.1]{ch22} \label{transfer lemma}
Let $\pi_1$ and $\pi_2$ be irreducible representations of $\G_n$ and $\G_{n-1}$ respectively. Then for any cuspidal representation $\sigma$ of $\G_2$ such that $\sigma \not \in \csuppline(\nu^{-1/2}\pi_1^{\vee})\cup \csuppline(\pi_2)$ we have that, \[ \Ext^i_{\G_{n-1}}(\pi_1,\pi_2^{\vee}) = \Ext^i_{\G_{n}}(\pi_2 \times \sigma,\pi_1^{\vee}) \] for all integers $i\geq 0$.  
\end{lemma}

\subsection{A Replacement Lemma}

We shall also require the following replacement lemma. It will be useful in the proof of Lemma \ref{lemma for the key reduction step} by allowing us to replace some factors in a restriction of an Arthur type representation by some cuspidal representation. 

\begin{lemma}\cite[Lemma 3.6]{ch22} \label{replacement lemma}
Let $\pi_1 \in \Alg(\G_k)$ and $\pi_2 \in \Alg(\G_l)$. Let $\pi_3 \in \Alg(\G_n)$ with $n\geq l+k$. Let $a=(n+1)-(k+l)$. Then, for any $\sigma$ in $\Irr^c(\G_a)$ such that $\sigma \not \in \csuppline(\nu^{-1/2}\pi_3)$, and for any $i$,
\[ \Ext^i_{\G_{n}}(\pi_1\times (\sigma\times \pi_2)|_{\G_{n-k}},\pi_3) = \Ext^i_{\G_{n}}(\pi_1\times (\Pi_a \btimes \pi_2)|_{\G_{n-k}},\pi_3) \]   
\end{lemma}

\section{Some Lemmas on Extensions on the Same Group} \label{section five}

In this Section, we prove some lemmas about extensions between representations of Arthur type on the same group. These lemmas about extensions on the same group are needed because we shall ultimately reduce the $\Ext$ branching problem to a study of extensions between representations of the same group.

\subsection{Extensions between Speh Representations}
We have the following lemma about $\Ext$ modules of Speh representations on the same group.

\begin{lemma} \label{lemma about extensions between speh representations on same group}
Let $\pi_1$ and $\pi_2$ be two Speh representations of $\G_n$ such that, \[ \Ext_{\G_n}^i(\pi_1,\pi_2)\neq 0 \] for some integer $i\geq 0$. Then either $\pi_2=\pi_1$ or $\pi_2=D(\pi_1)$. 
\end{lemma}
\begin{proof}

The proof follows from looking at the cuspidal supports. \end{proof}

\begin{remark}
The above lemma (proved by looking at cuspidal supports) is equivalent to the observation that if $V_1,V_2,V_3$ and $V_4$ are finite dimensional irreducible representations of $\SL_2(\mathbb{C})$ such that, \[ V_1\otimes V_2 = V_3\otimes V_4 \] then, $(V_3,V_4) = (V_1,V_2)$ or $(V_3,V_4) = (V_2,V_1)$.
\end{remark}

\begin{remark}
In the context of Lemma \ref{lemma about extensions between speh representations on same group} it would be interesting to obtain more precise information on Ext groups of Speh representations. It should be possible to study extensions between Speh representations based on a Koszul type resolution in the setting of graded Hecke algebras (see \cite{ch16}). 
\end{remark}

\subsection{Lemmas about extensions between products of unitary segment type representations}

Recall that by a unitary representation of segment type (see Section \ref{subsection on definition of speh representations of segment type}), we mean a unitary representation of the form $Z(\Delta)$ or $Q(\Delta)$ for some segment $\Delta$. In the proof of the lemma below we shall freely invoke Lemma \ref{lemma about jacquet modules of representations of segment type} and the observations in Remark \ref{remark on jacquet modules of representations of segment type} without citing them explicitly each time.

\begin{lemma} \label{lemma one about extensions between segment type representations}
Let $\pi_1$ be a unitary segment type representation of $\G_n$ and let $\tau\in \Alg(\G_k)$ (not necessarily of finite length). Let $\pi_2\in \Alg(\G_{n+k})$ be an irreducible unitary representation of Arthur type. Suppose that $\pi_2= \pi_{2,1}\times \pi_{2,2}\times \cdots \times \pi_{2,s}$ where, $\pi_{2,1},\pi_{2,2},\cdots,\pi_{2,s}$ are unitary representations of segment type. Suppose that,
\[ \Ext_{\G_{n+k}}^i(\pi_1\times \tau,\pi_2)\neq 0 \] for some integer $i\geq 0$. Then one of the following holds:
\begin{enumerate}
    \item $\pi_1 = \pi_{2,m}$ for some $m\in\{ 1,2,\cdots,s \}$. And, $\pi_2 = \pi_1\times \pi_2^{\prime}$ where $\pi_2^{\prime}\in \Alg(\G_k)$ is the product of the factors $\pi_{2,1},\cdots,$ $\pi_{2,m-1},$ $\pi_{2,m+1},\cdots,\pi_{2,s}$. Moreover, \[ \Ext_{\G_k}^j(\tau,\pi_2^{\prime})\neq 0 \] for some $j\leq i$. 
     \item $D(\pi_1) = \pi_{2,m}$ for some $m\in\{ 1,2,\cdots,s \}$. And, $\pi_2 = D(\pi_1)\times \pi_2^{\prime\prime}$ where $\pi_2^{\prime\prime}\in \Alg(\G_k)$ is the product of the factors $\pi_{2,1},\cdots,$ $\pi_{2,m-1},$ $\pi_{2,m+1},\cdots,\pi_{2,s}$. Moreover, \[ \Ext_{\G_k}^j(\tau,\pi_2^{\prime\prime})\neq 0 \] for some $j< i$. 
\end{enumerate}

\end{lemma}

\begin{proof}
We are given that, \[ \Ext_{\G_{n+k}}^i(\pi_1\times \tau,\pi_2)\neq 0. \] By second adjointness we have that, \begin{equation} \label{equation one for lemma two in section 5}
\Ext_{\G_n\times \G_k}^i(\pi_1\otimes \tau,\Omega)\neq 0
\end{equation}  for some irreducible subquotient $\Omega$ of $\ropp_{(n,k)}(\pi_2)$. Let us suppose that $\pi_1$ is of the form $Z(\Delta)$. The case when $\pi_1$ is of the form $Q(\Delta)$ can be handled similarly. Since $\pi_2$ is irreducible we arrange the factors of $\pi_2$ so that factors of the form $Z(\Delta)$ occur first followed by factors of the form $Q(\Delta)$.
Therefore, let us suppose that $\pi_1=Z(\Delta_0)$, $\pi_{2,i}=Z(\Delta_{2,i})$ ($i=1,2,\cdots,k$) and $\pi_{2,j}=Q(\Delta_{2,j})$ ($j=k+1,\cdots,s$). Here, $0\leq k \leq s$ is an integer and $\Delta_0,\Delta_{2,i}$ ($i=1,2,\cdots,k$) and $\Delta_{2,j}$ ($j=k+1,\cdots,s$) are some segments. Note that all the representations $\pi_{2,1},\pi_{2,2},\cdots,\pi_{2,s}$ and $\pi_1$ are unitary.

By the geometric lemma we obtain a filtration on $\ropp_{(n,k)}(\pi_2)$ where the subquotients are of the form, \[ (\omega_{1,1} \times \cdots \times\omega_{1,k} \times \xi_{1,k+1} \times\cdots \times \xi_{1,s}) \otimes (\omega_{2,1} \times \cdots \times \omega_{2,k} \times \xi_{2,k+1} \times \cdots  \times \xi_{2,s}), \] where $\omega_{1,i}\otimes \omega_{2,i}$ is the opposite Jacquet module of $\pi_{2,i}$ (here $i=1,2,\cdots,k$) with respect to a suitable opposite parabolic. Similarly, $\xi_{1,j}\otimes \xi_{2,j}$ is the opposite Jacquet module of $\pi_{2,j}$ (here $j=k+1,\cdots,s$) with respect to a suitable opposite parabolic.

We wish to analyse which of the above subquotients can make a non-zero contribution in Equation \ref{equation one for lemma two in section 5}. We carry out this analysis by considering the kinds of subquotients as in the following two cases.

\textbf{Case 1:} Suppose each of the factors $\omega_{1,1}, \omega_{1,2},\cdots,\omega_{1,k}$ is equal to the trivial representation of $\G_0=\langle e \rangle$. Let us suppose that a subquotient of this kind makes a non-zero contribution in Equation \ref{equation one for lemma two in section 5}. In this case by Equation \ref{equation one for lemma two in section 5} and the Kunneth formula we conclude that, \begin{equation}  \label{equation two for lemma two in section five}
\Ext_{\G_n}^p(\pi_1,\xi_{1,k+1} \times\cdots \times \xi_{1,s})\neq 0    
\end{equation} for some integer $p\leq i$.

Note that $e(\xi_{1,j})\leq 0$ for all $j=k+1,\cdots,s$ whereas, $e(\pi_1) = 0$. Therefore, by comparing central characters we conclude that for all $j=k+1,\cdots,s$ either $\xi_{1,j}=Q(\Delta_{2j})$ or $\xi_{1,j}$ is the trivial representation of $\G_0$. We now claim that $\xi_{1,j}=Q(\Delta_{2j})$ for exactly one $j\in \{k+1,\cdots,s\}$. 

Suppose for the sake of contradiction that $\xi_{1,j}=Q(\Delta_{2j})$ for two distinct integers $j_1$ and $j_2$. Without loss of generality we may take $j_1=k+1$ and $j_2=k+2$. By Equation \ref{equation two for lemma two in section five} we know that, \[ \Ext_{\G_n}^p(\pi_1,Q(\Delta_{2,k+1}) \times Q(\Delta_{2,k+2}) \times\zeta)\neq 0, \] where, $\zeta = \xi_{1,k+3} \times\cdots \times \xi_{1,s}$. Now by Frobenius reciprocity, \[ \Ext_{\G_l \times \G_{n-l}}^p(\eta_1\otimes \eta_2,Q(\Delta_{2,k+1}) \otimes (Q(\Delta_{2,k+2}) \times\zeta))\neq 0, \] where, $l=n(Q(\Delta_{2,k+1}))$ and $\eta_1\otimes \eta_2$ is the Jacquet module of $\pi_1$ with respect to a suitable parabolic. Hence we obtain that, \[ \Ext_{\G_l}^q(\eta_1,Q(\Delta_{2,k+1}))\neq 0, \] for some integer $q\leq p$. Now since $e(\eta_1)>0$ and $e(Q(\Delta_{2,k+1})) = 0$, by comparing central characters we obtain a contradiction.

Hence we conclude that in Case 1, the only subquotient of $\ropp_{(n,k)}(\pi_2)$ that can make a non-zero contribution in Equation \ref{equation one for lemma two in section 5} is of the form, \[ \Omega = Q(\Delta_{2j}) \otimes \pi^{\prime\prime} \] where, $\pi^{\prime\prime}$ is the product of segment type representations $\pi_{2,1},\pi_{2,2},\cdots,\pi_{2,s}$ except the term $Q(\Delta_{2j})$.

Now by Equation \ref{equation one for lemma two in section 5}, applying the Kunneth formula and Lemma \ref{lemma about extensions between speh representations on same group} proves part (2) of our lemma.

\textbf{Case 2:} Now suppose that atleast one of the factors $\omega_{1,1}, \omega_{2,1},\cdots,\omega_{1,k}$ is a representation of $\G_M$, where $M\geq 1$. Without loss of generality we may assume that $\omega_{1,1}$ is a representation of $\G_M$, where $M\geq 1$. Let us suppose that a subquotient of this kind makes a non-zero contribution in Equation \ref{equation one for lemma two in section 5}. Now by Equation \ref{equation one for lemma two in section 5} and the Kunneth formula we have that, \[
\Ext_{\G_n}^p(\pi_1,\omega_{1,1}\times \omega_{2,1},\times \cdots\times \omega_{1,k} \times \xi_{1,k+1} \times\cdots \times \xi_{1,s})\neq 0\] for some integer $p\leq i$.

Hence, by Frobenius reciprocity we have that, 
\[ \Ext_{\G_l \times \G_{n-l}}^p(\eta_1\otimes \eta_2,\omega_{1,1}\otimes (\omega_{2,1},\times \cdots\times \omega_{1,k} \times \xi_{1,k+1} \times\cdots \times \xi_{1,s}))\neq 0, \] where, $l=n(\omega_{1,1})$ and $\eta_1\otimes \eta_2$ is the Jacquet module of $\pi_1$ with respect to a suitable parabolic. Hence we have that, \begin{equation} \label{equation three for lemma two in section five}
\Ext_{\G_l}^q(\eta_1,\omega_{1,1})\neq 0,    
\end{equation}
for some integer $q\geq 0$. Let us suppose that $\pi_1= Z(\Delta_0)= $ $Z([-\frac{(x-1)}{2},\frac{(x-1)}{2}]_{\rho_1})$ and $\pi_{2,1} = Z(\Delta_{2,1})= Z([-\frac{(y-1)}{2},\frac{(y-1)}{2}]_{\rho_2})$, where $x,y$ are integers such that $x,y\geq 1$, and $\rho_1$ and $\rho_2$ are some unitary cuspidal representations.

Then $\eta_1 = Z([-\frac{(x-1)}{2},w]_{\rho_1})$ and $\omega_{1,1}= Z([z,\frac{(y-1)}{2}]_{\rho_2})$, where, $w,z$ are integers such that $w\leq \frac{(x-1)}{2}$ and $z\geq -\frac{(y-1)}{2}$. Now by comparing cuspidal supports in Equation \ref{equation three for lemma two in section five} we conclude that $\rho_1=\rho_2$, $w=\frac{(y-1)}{2}$ and $z=-\frac{(x-1)}{2}$. Substituting these in the above inequalities we conclude that $x=y$. Hence $\eta_1= \pi_1$ and $\omega_{1,1}= \pi_{2,1}\cong \pi_1$.

Hence we conclude that in Case 2, the only subquotient of $\ropp_{(n,k)}(\pi_2)$ that can make a non-zero contribution in Equation \ref{equation one for lemma two in section 5} is of the form, \[ \Omega = Z(\Delta_{2i}) \otimes \pi^{\prime}, \] where, $\pi^{\prime\prime}$ is the product of segment type representations $\pi_{2,1},\pi_{2,2},\cdots,\pi_{2,s}$ except the term $Z(\Delta_{2i})$. Now using Equation \ref{equation one for lemma two in section 5} and applying the Kunneth formula proves part (1) of our lemma. \end{proof}

\begin{remark}
In the above lemma, the hypothesis that $\pi_1$ and $\pi_{2,1},\pi_{2,2},$ $\cdots,\pi_{2,s}$ are unitary representations of segment type cannot be dropped. For instance, the above lemma does not hold if we assume $\pi_1$ to be an arbitrary Speh representation.

Consider the following example. Let $\rho$ denote the trivial representation of $\G_1$. Let $\pi_1 = Q([0,1]_{\rho},[-1,0]_\rho)$ and $\pi_2 =  \rho \times Q([-1,0,1]_\rho)$. We claim that, \[ \Ext^1_{\G_4}(\pi_1,\pi_2) \neq 0. \] By Frobenius reciprocity we have that, \[ \Ext^i_{\G_4}(\pi_1,\pi_2) = \Ext^i_{\G_1\times \G_3}(\rho \otimes Q([0,1]_{\rho},[-1]_\rho),\rho \otimes  Q([-1,0,1]_\rho) ) \neq 0. \] Here, we have used the description of Jacquet modules of ladder representations in \cite{kl12}. By comparing cuspidal supports $\rho \otimes Q([0,1]_{\rho},[-1]_\rho)$ is the only subquotient of the $r_{(1,3)}(\pi_1)$ that can contribute to the above $\Ext$ module. By \cite{or05} we know that,  \[ \Ext^1_{\G_3}(Q([0,1]_{\rho},[-1]_\rho),Q([-1,0,1]_\rho))\neq 0.\] Applying the Kunneth formula we conclude that, \[ \Ext^1_{\G_4}(\pi_1,\pi_2)\neq 0. \]
\end{remark}

\begin{remark}
The above lemma can be thought of as an $\Ext$ analogue of \cite[Proposition 4.2]{ch22} albeit with much more restrictive hypothesis. The Proposition 4.2 in \cite{ch22} follows from \cite[Theorem 9.1]{ch22} (see \cite{ch24product} for a generalization of \cite[Theorem 9.1]{ch22}).
\end{remark}

The following lemma is like a converse to Lemma \ref{lemma one about extensions between segment type representations}. The statement of the lemma below is obvious if one takes $\Hom$ instead of $\Ext$ modules.

\begin{lemma} \label{lemma two about extensions between segment type representations}
 Let $\pi_1$ be a unitary segment type representation of $\G_n$ and let $\tau\in \Alg(\G_k)$ (not necessarily of finite length). Let $\sigma \in \Alg(\G_{k})$ be an irreducible unitary representation of Arthur type. Suppose that $\sigma = \sigma_1 \times \sigma_2 \times \cdots \times \sigma_s$ where, $\sigma_1,\sigma_2,\cdots,\sigma_s$ are unitary representations of segment type. Suppose that, \[ \Ext_{\G_{k}}^j(\tau,\sigma)\neq 0 \] for some integer $j\geq 0$. Then
\[ \Ext_{\G_{n+k}}^i(\pi_1\times \tau,\pi_1 \times  \sigma)\neq 0 \] for some integer $i\geq 0$. 
\end{lemma}

\begin{proof}
By second adjointness we have that, \begin{equation} \label{equation one for lemma three in section five}
    \Ext_{\G_{n+k}}^i(\pi_1\times \tau,\pi_1 \times  \sigma) = \Ext_{\G_{n}\times \G_k}^i(\pi_1\otimes \tau,\bar{r}_P(\pi_1 \times  \sigma)).
\end{equation} Here $\bar{r}_P$ denotes the opposite Jacquet module with respect to the opposite parabolic corresponding to the partition $(n,k)$. By the analysis carried out in the proof of Lemma \ref{lemma one about extensions between segment type representations}, the only subquotients of $\bar{r}_P(\pi_1\times \sigma)$ that can make a non-zero contribution to the $\Ext$ module in the right hand side of Equation \ref{equation one for lemma three in section five} are of the following two kinds:
\begin{enumerate}
    \item[(A)] $\Omega_1 = \pi_1 \otimes \sigma$. Note that by the discussion in Section \ref{section on geometric lemma}, $\Omega_1$ occurs as a \textit{submodule} of $\ropp_{P}(\pi_1 \times \sigma)$.
    \item[(B)] $\Omega_2 = D(\pi_1)\otimes \sigma^{\prime}$. Note that if $\Omega_2$ occurs as a subquotient of $\ropp_{P}(\pi_2)$ then there exist $m \in \{1,2,\cdots,s\}$ such that $\sigma_m= D(\pi_1)$. Here, $\sigma^{\prime} = \pi_1 \times \sigma_1 \times \cdots \times \sigma_{m-1} \times \sigma_{m+1}\times \cdots \times \sigma_s$. In this case we have that $\pi_1 \times \sigma = D(\pi_1)\times \sigma^{\prime}$ and once again by the discussion in Section \ref{section on geometric lemma}, we can ensure that $\Omega_2$ occurs as a \textit{submodule} of $\ropp_{P}(\pi_1 \times \sigma)$.
\end{enumerate}

We are given that \begin{equation} \label{equation two for lemma three in section five}
    \Ext_{\G_{k}}^j(\tau,\sigma)\neq 0
\end{equation} for some integer $j\geq 0$.
Also, we know that \begin{equation} \label{equation three for lemma three in section five}
    \Hom_{\G_{n}}(\pi_1,\pi_1)\neq 0.
\end{equation} Hence by Equation \ref{equation two for lemma three in section five}, Equation \ref{equation three for lemma three in section five} and the Kunneth formula we conclude that, \begin{equation} \label{equation four for lemma three in section five}
    \Ext_{\G_n \times \G_{k}}^j(\pi_1\otimes \tau,\pi_1\otimes \sigma)\neq 0.
\end{equation} We define the following sets. For $t=1,2$ we define, \[ S_t = \{ h\in \mathbb{Z} | \Ext_{\G_n \times \G_{k}}^h(\pi_1\otimes \tau,\Omega_t)\neq 0  \}. \] Let us denote $S = S_1 \cup S_2$. Then by Equation \ref{equation four for lemma three in section five} it is evident that the set $S$ is non-empty. Let $i^*$ denote the smallest integer in the set $S$. Then $\Ext_{\G_n \times \G_{k}}^{i^*}(\pi_1\otimes \tau,\Omega_{t^*})\neq 0$ where $t^*\in \{1, 2\}$.

Since $\Omega_{t^*}$ occurs as a \textit{submodule} of $\ropp_{P}(\pi_1 \times \sigma)$ by a long exact sequence argument we get that, \[ \Ext_{\G_{n}\times \G_k}^{i^*}(\pi_1\otimes \tau,\bar{r}_P(\pi_1 \times  \sigma)) \neq 0. \] Hence by Frobenius reciprocity we conclude that, \begin{equation*}
\Ext_{\G_{n+k}}^{i^*}(\pi_1\times \tau,\pi_1 \times  \sigma)\neq 0. \qedhere \end{equation*}  
\end{proof}

We also have the following similar lemma. 

\begin{lemma} \label{lemma three about extensions between segment type representations}
Let $\pi_1$ be a unitary segment type representation of $\G_n$ and let $\tau\in \Alg(\G_k)$ (not necessarily of finite length). Let $\sigma \in \Alg(\G_{k})$ be an irreducible unitary representation of Arthur type. Suppose that $\sigma = \sigma_1 \times \sigma_2 \times \cdots \times \sigma_s$ where, $\sigma_1,\sigma_2,\cdots,\sigma_s$ are unitary representations of segment type. Suppose that, \[ \Ext_{\G_{k}}^j(\tau,\sigma)\neq 0 \] for some integer $j\geq 0$. Then
\[ \Ext_{\G_{n+k}}^i(\pi_1\times \tau,D(\pi_1) \times  \sigma)\neq 0 \] for some integer $i\geq 0$. 
\end{lemma}

The proof of the above lemma is along similar lines as Lemma \ref{lemma two about extensions between segment type representations} except that we use the fact that $\Ext_{\G_{n}}^*(\pi_1,D(\pi_1))\neq 0$ instead of Equation \ref{equation three for lemma three in section five} in the proof of the Lemma \ref{lemma two about extensions between segment type representations}. The fact that $\Ext_{\G_{n}}^*(\pi_1,D(\pi_1))\neq 0$ is a simple consequence of the duality theorem (Theorem \ref{duality theorem of nori and prasad}) of Nori and Prasad.

\subsection{A reformulation}

We can reformulate Lemmas \ref{lemma one about extensions between segment type representations}, \ref{lemma two about extensions between segment type representations} and \ref{lemma three about extensions between segment type representations} into a single statement in the language of Arthur parameters as follows. Let $\Delta \SL_2$ denote the diagonally embedded $\SL_2$ sitting inside $\SL_2 \times \SL_2$. Note that if $\pi_1$ and $\pi_2$ are irreducible Arthur type representations of $\G_n$ then, $\A(\pi_1)|_{W_F\times \Delta \SL_2} \cong \A(\pi_2)|_{W_F\times \Delta \SL_2}$ if and only if $\pi_1$ and $\pi_2$ have the same cuspidal support.

\begin{lemma} \label{lemma on reformulation}
Let $\pi_1$ be a unitary segment type representation of $\G_n$ and let $\tau\in \Alg(\G_k)$ (not necessarily of finite length). Let $\pi_2\in \Alg(\G_{n+k})$ be an irreducible unitary representation of Arthur type. Suppose that $\pi_2= \pi_{2,1}\times \pi_{2,2}\times \cdots \times \pi_{2,s}$ where, $\pi_{2,1},\pi_{2,2},\cdots,\pi_{2,s}$ are unitary representations of segment type. Then,
\[ \Ext_{\G_{n+k}}^*(\pi_1\times \tau,\pi_2)\neq 0 \] if and only if $\pi_2$ decomposes as, $\pi_2 = \pi_2^{\prime}\times \pi_2^{\prime \prime}$ where, $\pi_2^{\prime}\in \Alg(\G_{n})$ and $\pi_2^{\prime\prime}\in \Alg(\G_{k})$ such that, \[ \A(\pi_1)|_{W_F\times \Delta \SL_2} \cong \A(\pi_2^{\prime})|_{W_F\times \Delta \SL_2} \] and, \[ \Ext_{\G_{k}}^*(\tau,\pi_2^{\prime \prime})\neq 0. \]
\end{lemma}

\subsection{A conjecture about extensions on the same group}

Since we are dealing with questions about extensions on the same group in this section, it is natural to ask when non-trivial extensions exist between Arthur type representations on the same group. The following conjecture arose in an email correspondence between Professor Dipendra Prasad and Professor K.Y. Chan.

\begin{conjecture} \label{conjecture on extensions on same group}
    
Suppose $\pi_1$ and $\pi_2$ are irreducible Arthur type representations of $\GL_n(F)$. Then $\Ext^*_{\GL_n(F)}(\pi_1,\pi_2)\neq 0$ if and only if \[ \A(\pi_1)|_{W_F\times \Delta \SL_2} \cong \A(\pi_2)|_{W_F\times \Delta \SL_2}. \]

\end{conjecture} 

\begin{remark} \label{remark on conjecture on extensions on same group}
Note that as observed earlier, saying that \[ \A(\pi_1)|_{W_F\times \Delta \SL_2} \cong \A(\pi_2)|_{W_F\times \Delta \SL_2} \] is the same as saying that $\pi_1$ and $\pi_2$ have the same cuspidal support.
\end{remark}

By repeated applications of Lemma \ref{lemma on reformulation} it follows that the above conjecture is true in the following special case.

\begin{proposition} 
Suppose $\pi_1$ and $\pi_2$ are irreducible Arthur type representations of $\GL_n(F)$. Moreover, suppose that $\pi_1$ and $\pi_2$ are products of unitary representations of segment type. Then $\Ext^*_{\GL_n(F)}(\pi_1,\pi_2)\neq 0$ if and only if \[ \A(\pi_1)|_{W_F\times \Delta \SL_2} \cong \A(\pi_2)|_{W_F\times \Delta \SL_2}. \] 
\end{proposition} 

\subsubsection{An Example} The Conjecture \ref{conjecture on extensions on same group} and Remark \ref{remark on conjecture on extensions on same group} suggests the following more general question. Suppose that $\pi_1$ and $\pi_2$ are irreducible representation of $\GL_n(F)$ having the same cuspidal support. Then is it true that, \[ \Ext^i_{\GL_n(F)}(\pi_1,\pi_2)\neq 0 \] for some integer $i\geq 0$? The answer to the above question is no in general. There exist irreducible representations $\pi_1$ and $\pi_2$ of $\GL_n(F)$ having the same cuspidal support for which $\Ext^*_{\GL_n(F)}(\pi_1,\pi_2) = 0$. We owe the following example to Professor K.Y. Chan. 

Consider the representations $\pi_1$ and $\pi_2$ of $\GL(4)$ defined as follows. \[ \pi_1 = Q([2,3],[2],[1]) \text{ and } \pi_2 = Q([1,3],[2]).\] We claim that for all integers $i\geq 0$, \[ \Ext^i_{\GL_4(F)}(\pi_1,\pi_2)= 0. \] It is not difficult to see that $\pi_1$ sits in the following short exact sequence, \begin{equation} \label{equation one for subsubsection an example}
 0 \rightarrow Q([2,3]) \times Q([1,2]) \rightarrow Q([2,3]) \times Q([2]) \times Q([1]) \rightarrow \pi_1 \rightarrow 0.   
\end{equation} By Frobenius reciprocity we have that, \[ \Ext^i_{\GL(4)}(Q([2,3]) \times Q([1,2]),\pi_2) = \Ext^i_{\GL(2) \times \GL(2)}(Q([2,3])\otimes Q([1,2]),\Bar{r}_{(2,2)}(\pi_2)). \] Here, $\Bar{r}_{(2,2)}(\pi_2)$ denotes the opposite Jacquet module of $\pi_2$ with respect to the the opposite parabolic corresponding to the partition $(2,2)$. 

Let $\Omega = \omega_1 \otimes \omega_2$ be any irreducible subquotient of $\Bar{r}_{(2,2)}(\pi_2)$. By the geometric lemma we conclude that $\nu$ always belongs to the cuspidal support of $\omega_1$. Hence using Kunneth formula and comparing cuspidal supports we obtain that, \[ \Ext^i_{\GL(2) \times \GL(2)}(Q([2,3]\otimes Q[1,2]),\Bar{r}_{(2,2)}(\pi_2)) = 0. \] Hence we conclude that for all integers $i\geq 0$, \begin{equation} \label{equation two for subsubsection an example} \Ext^i_{\GL_4(F)}(Q([2,3]) \times Q([1,2]),\pi_2) = 0. \end{equation} Similarly we can argue that \begin{equation} \label{equation three for subsubsection an example} \Ext^i_{\GL_4(F)}(Q([2,3]) \times Q([2]) \times Q([1]),\pi_2) = 0 \end{equation} for all integers $i\geq 0$. 

Then from Equations \ref{equation one for subsubsection an example}, \ref{equation two for subsubsection an example} and \ref{equation three for subsubsection an example} we conclude that,  \[ \Ext^i_{\GL_4(F)}(\pi_1,\pi_2)= 0 \] for all integers $i\geq 0$.

\section{Main Reduction Lemma} \label{section six}

The following lemma is Lemma 4.3 from \cite{ch22} where it is stated for $\Hom$ spaces but which works as well for $\Ext$ modules as we assert below. This lemma follows from an application of the filtration stated in Theorem \ref{filtration for parabolically induced modules}. It allows us to reduce the study of $\Ext$ branching to the bottom piece of the aforementioned filtration. Since this lemma plays an important role in our arguments we recall its proof for the sake of completeness.

\begin{lemma} \label{lemma for the key reduction step}
Let $\pi_1$ and $\pi_2$ be Arthur type representations of $\G_n$ and $\G_{n-1}$ respectively. Suppose that \[ \pi_1= u_{\rho_1}(a_1,b_1)\times u_{\rho_2}(a_2,b_2)\times \cdots \times u_{\rho_r}(a_r,b_r)\] and \[ \pi_2 = u_{\tau_1}(c_1,d_1)\times u_{\tau_2}(c_2,d_2)\times \cdots \times u_{\tau_s}(c_s,d_s)\] where $\rho_i$ ($i=1,2,\cdots,r$) and $\tau_j$ ($j=1,2,\cdots,s$) are unitary cuspidal representations. Further suppose that $a_1+b_1 \geq a_i + b_i$ and $a_1+b_1 \geq c_j + d_j$  for all $i=1,2,\cdots,r$ and $j=1,2,\cdots,s$. Let $\sigma$ be a unitary cuspidal representation of $\G_{a_1n(\rho_1)}$ such that $\nu^{1/2}\sigma\not\in \csupp_{\mathbb{Z}}(\pi_2)$ and $\sigma\not\in \csupp_{\mathbb{Z}}(\pi_1)$. Then for all integers $i\geq 0$ we have, \[ \Ext^i_{\G_{n-1}}(\pi_1,\pi_2) \cong \Ext^i_{\G_{n-1}} (u_{\rho_1}(a_1,b_1-1)\times (\sigma \times \pi_{1}^{\prime})|_{\G_{a}},\pi_2)  \] where, $a=n - n(\rho_1)a_1(b_1-1)-1$ and $\pi_{1}^{\prime} = u_{\rho_2}(a_2,b_2)\times \cdots \times u_{\rho_r}(a_r,b_r)$.

\end{lemma}

\begin{proof}

By Theorem \ref{filtration for parabolically induced modules} we have that $\pi_1|_{\G_{n-1}}$ is glued together from the pieces, 

\[ V_0/V_1 \cong \nu^{1/2}u_{\rho_1}(a_1,b_1)\times\pi_{1}^{\prime}|_{\G_{n(\pi_{1}^{\prime})-1}},\] for $1\leq k\leq d$, 
\[ V_k/V_{k+1} \cong \nu^{1/2}u_{\rho_1}(a_1,b_1)^{(k)}\times (\Pi_k \btimes \pi_{1}^{\prime})|_{\G_{n(\pi_{1}^{\prime}) + k -1}}.\] Here $d=\lev(u_{\rho_1}(a_1,b_1))$ is the largest integer $\ell$ such that $u_{\rho_1}(a_1,b_1)^{(\ell)}\neq 0$. By Lemma \ref{derivative of speh representations}, we know that $d=n(\rho_1)a_1$.

\textbf{Claim:} We claim that among all the pieces in the above filtration the pieces other than the last piece do not contribute to the $\Ext$ module, that is for all integers $i\geq 0$, \[   \Ext^i_{\G_{n-1}} (V_k/V_{k+1},\pi_2) = 0 \] for all $0\leq k < d$. 

\textbf{Proof of the Claim:} The proof of the claim is based on a simple observation about the cuspidal support of  derivatives of Speh representations. If $S = u_{\rho}(m,n)$ is a Speh representation then $\nu^{(m+n-2)/2}\rho \in \csupp(S^{(k)})$ provided $k < \lev(S)$, that is if $S^{(k)}$ is not the highest derivative of $S$ (Lemma \ref{derivative of speh representations}). Suppose that $k< d = \lev(u_{\rho_1}(a_1,b_1))$. Now by definition \[  \Ext^i_{\G_{n-1}} (V_k/V_{k+1},\pi_2) = \Ext^i_{\G_{n-1}} (\nu^{1/2}u_{\rho_1}(a_1,b_1)^{(k)} \times N_k ,\pi_2) \] where $N_k = (\Pi_k \btimes \pi_{1}^{\prime})|_{\G_{n(\pi_{1}^{\prime}) + k -1}}$. The inequalities $a_1+b_1 \geq a_i + b_i$ and $a_1+b_1 \geq c_j + d_j$ imply that $1/2 + (a_1 + b_1 - 2)/2$ is \textit{strictly} greater than any element of the set, 
\[ S = \{  (a_i + b_i - 2)/2 ,  (c_j + d_j - 2)/2 | i=1,2,\cdots,r ; j=1,2,\cdots,s \}. \] So if $k < \lev(u_{\rho_1}(a_1,b_1))$ we conclude that $\nu^{1/2  + (a_1 + b_1 - 2)/2}\rho_1$ lies in $ \csupp(u_{\rho_1}(a_1,b_1)^{(k)})$ but does not lie in the cuspidal support of $\pi_2$. Therefore by second adjointness and comparing cuspidal supports at $\nu^{1/2  + (a_1 + b_1 - 2)/2}\rho_1$ (see Lemma \ref{lemma about standard argument on ext vanishing}) we conclude that, \[   \Ext^i_{\G_{n-1}} (V_k/V_{k+1},\pi_2) = 0 \] for all $0\leq k < d$. This completes the proof of the claim.

Since the claim is true, a long exact sequence argument implies that for all integers $i\geq 0$, \begin{align*}
\Ext^i_{\G_{n-1}}(\pi_1,\pi_2) &= \Ext^i_{\G_{n-1}}(V_d/V_{d+1},\pi_2) \\ &= \Ext^i_{\G_{n-1}} (u_{\rho_1}(a_1,b_1)^{-}\times (\Pi_{d} \btimes \pi_{1}^{\prime})|_{\G_a},\pi_2).    
\end{align*} Here, $u_{\rho_1}(a_1,b_1)^{-}=\nu^{1/2}u_{\rho_1}(a_1,b_1)^{(d)} = u_{\rho_1}(a_1,b_1-1)$. Since $\nu^{1/2}\sigma\not\in \csupp_{\mathbb{Z}}(\pi_2)$ and $\sigma\not\in \csupp_{\mathbb{Z}}(\pi_1)$ by Lemma \ref{replacement lemma} it follows that,
\begin{equation*}
 \Ext^i_{\G_{n-1}}(\pi_1,\pi_2) = \Ext^i_{\G_{n-1}} (u_{\rho_1}(a_1,b_1) \times (\sigma \times \pi_{1}^{\prime})|_{\G_{a}},\pi_2). \qedhere \end{equation*}\end{proof}

\section{Proof of one direction of Theorem \ref{main theorem for segment type representations}}  \label{section 7 on proof of one direction}

In this Section we shall prove one direction of Theorem \ref{main theorem for segment type representations}. More precisely, we prove the following.

\begin{theorem} \label{one direction of main theorem}
Let $n\geq 1$ be an integer. Let $\pi_1$ and $\pi_2$ be irreducible Arthur type representations of $\GL_n(F)$ and $\GL_{n-1}(F)$ respectively. Suppose that $\pi_1$ and $\pi_2$ are products of discrete series representations and the Aubert-Zelevinsky duals of discrete series representations. If $\Ext^i_{\GL_{n-1}(F)}(\pi_1,\pi_2)\neq 0 $ for some integer $i\geq 0$, then $\pi_1$ and $\pi_2$ are strong $\Ext$ relevant.
\end{theorem}

\begin{proof}

Suppose $\pi_1= \pi_{1,1}\times \pi_{1,2}\times \cdots \times \pi_{1,r}$ and $\pi_2= \pi_{2,1}\times \pi_{2,2}\times \cdots \times \pi_{2,s}$. Let us denote $\pi_{1,i}=u_{\rho_i}(a_i,b_i)$ and $\pi_{2,j}=u_{\tau_j}(c_i,d_i)$ where the $\rho_i$ ($i=1,2,\cdots,r$) and $\tau_j$ ($j=1,2,\cdots,s$) are some cuspidal representations. We are given that $\pi_{1,1},\pi_{1,2},\cdots,\pi_{1,r}$ and $\pi_{2,1},\pi_{2,2},\cdots,\pi_{2,s}$ are all of segment type. This means that for all $i=1,2,\cdots,r$ either $a_i=1$ or $b_i=1$. Similarly, for all $j=1,2,\cdots,s$ either $c_j=1$ or $d_j=1$. We know that $\Ext^i_{\G_{n-1}}(\pi_1,\pi_2)\neq 0$ for some integer $i\geq 0$. We want to show that the pair $(\pi_1,\pi_2)$ is strong $\Ext$ relevant. 

Let $m(\pi_1,\pi_2)$ denote the number of non-cuspidal factors of $\pi_1$ and $\pi_2$. The proof of the Theorem is via induction on $m(\pi_1,\pi_2)$. When $m(\pi_1,\pi_2)=0$ then both $\pi_1$ and $\pi_2$ are generic and hence automatically strong $\Ext$ relevant. We now assume that Theorem \ref{one direction of main theorem} holds for all integers less than $m(\pi_1,\pi_2)$.

Since $\pi_1$ and $\pi_2$ are irreducible we can permute their factors and ensure that one of the following cases occurs. We can ensure that either:
\begin{enumerate}
    \item $a_1+b_1\geq a_i+b_i$ and $a_1+b_1\geq c_j+d_j$ for all $i$ and $j$.
    \item $c_1+d_1\geq c_j+d_j$ and $c_1+d_1\geq a_i+b_i$ for all $i$ and $j$.
\end{enumerate}
We deal with the above two cases one by one.

\textbf{Case 1:} Let us suppose that $a_1+b_1\geq a_i+b_i$ and $a_1+b_1\geq c_j+d_j$ for all $i$ and $j$. Choose a unitary cuspidal representation $\sigma$ of $\G_2$ such that $\nu^{1/2}\sigma\not \in \csuppline(\pi_2)$ and $\sigma\not \in \csuppline(\pi_1)$. 

Since $\Ext^i_{\G_{n-1}}(\pi_1,\pi_2)\neq 0$ by Lemma \ref{lemma for the key reduction step} we have that,\[ \Ext^i_{\G_{n-1}} (u_{\rho_1}(a_1,b_1-1)\times (\sigma \times \pi_{1}^{\prime})|_{\G_{a}},\pi_2)\neq 0  \] where, $a=n - n(\rho_1)a_1(b_1-1)-1$ and $\pi_{1}^{\prime} = u_{\rho_2}(a_2,b_2)\times \cdots \times u_{\rho_r}(a_r,b_r)$. Note that $\pi_{1,1}^{-} = u_{\rho_1}(a_1,b_1)^- = u_{\rho_1}(a_1,b_1-1)$ is also of segment type.

We now invoke Lemma \ref{lemma one about extensions between segment type representations} and suitably rearrange the factors of $\pi_2$ to conclude the following.
\begin{enumerate}
    \item $\pi_{2,1}$ is equal to either $\pi_{1,1}^-$ or $D(\pi_{1,1}^-)$. Therefore, the pair $(\pi_{1,1},\pi_{2,1})$ is strong $\Ext$ relevant.
    \item We have that $\Ext^i_{\G_{a}} (\sigma \times \pi_{1}^{\prime},\pi_{2,2}\times \cdots \pi_{2,s})\neq 0$. So by the induction hypothesis the pair $(\sigma \times \pi_{1}^{\prime},\pi_{2,2}\times \cdots \pi_{2,s})$ is strong $\Ext$ relevant. Since $\nu^{1/2}\sigma\not \in \csuppline(\pi_2)$, the pair $(\pi_{1}^{\prime},\pi_{2,2}\times \cdots \pi_{2,s})$ is strong $\Ext$ relevant. 
\end{enumerate}

From the above two observations we conclude that the pair $\pi_1$ and $\pi_2$ are strong $\Ext$ relevant.

\textbf{Case 2:} Let us suppose that $c_1+d_1\geq c_j+d_j$ and $a_1+b_1\geq c_i+d_i$ for all $i$ and $j$. Then we can apply the transfer lemma (Lemma \ref{transfer lemma}) and deal with the pair $(\pi_2^{\vee}\times \tilde{\sigma},\pi_1^{\vee})$ rather than the pair $(\pi_1,\pi_2)$ where, $\tilde{\sigma}$ is some cuspidal representation of $\G_2$ chosen such that $\nu^{1/2}\tilde{\sigma} \not \in \csuppline(\pi_1^{\vee})$ and $\tilde{\sigma} \not \in \csuppline(\pi_2^{\vee})$. Since, $\tilde{\sigma} \not \in \csuppline(\nu^{-1/2}\pi_1^{\vee})$ observe that the pair $(\pi_1,\pi_2)$ is strong $\Ext$ relevant if and only if the pair $(\pi_2^{\vee} \times \tilde{\sigma},\pi_1^{\vee})$ is strong $\Ext$ relevant. When dealing with the pair $(\pi_2^{\vee} \times \tilde{\sigma},\pi_1^{\vee})$ the role of $\pi_{1,1}$ is replaced by $\pi_{2,1}$. Hence it is enough to deal with Case 1.
\end{proof}

\section{Proof of other direction of Theorem \ref{main theorem for segment type representations}}  \label{section 8 on proof of other direction}

We now prove the other direction of Theorem \ref{main theorem for segment type representations}. 

\begin{theorem} \label{other direction of main theorem}
Let $\pi_1$ and $\pi_2$ be irreducible Arthur type representations of $\GL_n(F)$ and $\GL_{n-1}(F)$ respectively. Suppose that $\pi_1$ and $\pi_2$ are products of discrete series representations and the Aubert-Zelevinsky duals of discrete series representations. If $\pi_1$ and $\pi_2$ are strong $\Ext$ relevant then, $\Ext^i_{\GL_{n-1}(F)}(\pi_1,\pi_2)\neq 0 $ for some integer $i\geq 0$. 
\end{theorem}

\begin{proof}
We set up notations as in the previous section. Suppose $\pi_1= \pi_{1,1}\times \pi_{1,2}\times \cdots \times \pi_{1,r}$ and $\pi_2= \pi_{2,1}\times \pi_{2,2}\times \cdots \times \pi_{2,s}$, where $\pi_{1,1},\pi_{1,2},\cdots,\pi_{1,r}$ and $\pi_{2,1},\pi_{2,2},\cdots,\pi_{2,s}$ are unitary representations of segment type. Let us denote $\pi_{1,i}=u_{\rho_i}(a_i,b_i)$ and $\pi_{2,j}=u_{\tau_j}(c_i,d_i)$ where the $\rho_i$ ($i=1,2,\cdots,r$) and $\tau_j$ ($j=1,2,\cdots,s$) are cuspidal representations.  

We are given that $\pi_1$ and $\pi_2$ are strong $\Ext$ relevant. Let $m(\pi_1,\pi_2)$ denote the number of non-cuspidal factors of $\pi_1$ and $\pi_2$. As before, the proof of the Theorem is via induction on $m(\pi_1,\pi_2)$. 
When $m(\pi_1,\pi_2)=0$ then both $\pi_1$ and $\pi_2$ are generic and by Theorem \ref{generic branching law} we conclude that $\Hom_{\G_{n-1}}(\pi_1,\pi_2)\neq 0$. We now assume that the Theorem \ref{other direction of main theorem} holds for all integers less than $m(\pi_1,\pi_2)$.

Arguing similarly as in Case 2 of the proof of Theorem \ref{one direction of main theorem}, it is enough to deal with the case that $a_1+b_1\geq a_i+b_i$ and $a_1+b_1\geq c_j+d_j$ for all $i$ and $j$. Choose a unitary cuspidal representation $\sigma$ of $\G_2$ such that $\nu^{1/2}\sigma\not \in \csuppline(\pi_2)$ and $\sigma\not \in \csuppline(\pi_1)$. Then by Lemma \ref{lemma for the key reduction step} we have that for all integers $i\geq 0$, \begin{equation} \label{equation one in proof of if direction}
\Ext^i_{\G_{n-1}}(\pi_1,\pi_2) = \Ext^i_{\G_{n-1}} (\pi_{1,1}^{-}\times (\sigma \times \pi_{1}^{\prime})|_{\G_{a}},\pi_2)  \end{equation} where, $a=n - n(\rho_1)a_1(b_1-1)-1$ and $\pi_{1}^{\prime} = u_{\rho_2}(a_2,b_2)\times \cdots \times u_{\rho_r}(a_r,b_r)$. Here $\pi_{1,1}^{-} = u_{\rho_1}(a_1,b_1)^- = u_{\rho_1}(a_1,b_1-1)$.

Now since the representations $\pi_1$ and $\pi_2$ are strong $\Ext$ relevant one of the following two cases must occur. \begin{enumerate}
    \item [(A)] $\pi_{2,m} = \pi_{1,1}^-$ for some $m=1,2,\cdots,s$. In this case $\pi_2 = \pi_{1,1}^- \times \pi_2^{\prime}$ where, $\pi_2^{\prime} = \pi_{2,1} \times\cdots \times \pi_{2,m-1} \times \pi_{2,m+1} \times \cdots \times \pi_{2,s}$. In this case, the representations $\pi_{1}^{\prime}$ and $\pi_2^{\prime}$ are strong $\Ext$ relevant.
    \item [(B)] $\pi_{2,m} = D(\pi_{1,1}^-)$ for some $m=1,2,\cdots,s$. In this case $\pi_2 = D(\pi_{1,1}^-) \times \pi_2^{\prime\prime}$ where, $\pi_2^{\prime\prime} = \pi_{2,1} \times\cdots \times \pi_{2,m-1} \times \pi_{2,m+1} \times \cdots \times \pi_{2,s}$. In this case, the representations $\pi_{1}^{\prime}$ and $\pi_2^{\prime\prime}$ are strong $\Ext$ relevant.
\end{enumerate}

Let us suppose that Case (A) above occurs, that is, $\pi_2 = \pi_{1,1}^- \times \pi_2^{\prime}$, and the representations $\pi_{1}^{\prime}$ and $\pi_2^{\prime}$ are strong $\Ext$ relevant. Since $\pi_{1}^{\prime}$ and $\pi_2^{\prime}$ are strong $\Ext$ relevant we have that the representations $\sigma\times \pi_{1}^{\prime}$ and $\pi_2^{\prime}$ are strong $\Ext$ relevant. By the induction hypothesis we conclude that, \[ \Ext^*_{\G_a}(\sigma\times \pi_{1}^{\prime},\pi_2^{\prime})\neq 0. \] By Lemma \ref{lemma two about extensions between segment type representations} and Equation \ref{equation one in proof of if direction} we conclude that \[ \Ext^i_{\G_{n-1}}(\pi_1,\pi_2)\neq 0 \] for some integer $i\geq 0$.

On the other hand suppose that Case (B) occurs. In this case, $\pi_2 = D(\pi_{1,1}^-) \times \pi_2^{\prime\prime}$, and the representations $\pi_{1}^{\prime}$ and $\pi_2^{\prime\prime}$ are strong $\Ext$ relevant. The proof proceeds exactly as in Case (A) but we invoke Lemma \ref{lemma three about extensions between segment type representations} instead of Lemma \ref{lemma two about extensions between segment type representations}.
\end{proof}

\section{A Final Remark} \label{section nine}

In this section we make a remark about the non-uniqueness of matching of factors in Definition \ref{definition of strong ext relevance} of strong Ext relevance. We first discuss the uniqueness of matching in the context of the original definition of relevance in the work \cite{ggp20}.

Let $m, n \in \mathbb{Z}_{\geq 0}$ and let $\pi_1$ and $\pi_2$ be Arthur type representations of $\GL_m(F)$ and $\GL_n(F)$ respectively. Let $\A(\pi_1)$ and $\A(\pi_2)$ denote their respective Arthur parameters. Recall that in the work \cite{ggp20}, the authors define a pair ($\pi_1$,$\pi_2$) to be relevant if there exist admissible homomorphisms $\{\phi_i\}_{i=1}^{i=r+s}$ of $W_F$ with bounded image and positive integers $a_1,a_2,\cdots a_r, b_{r+1},b_2,$ $\cdots b_{r+s},$ $ c_1,c_2,\cdots c_{r+s}$ such that,
\[ \A(\pi_1) = \sum_{i=1}^{r} \phi_i \otimes V_{c_i} \otimes V_{a_i} \oplus \sum_{i=r+1}^{r+s} \phi_i \otimes V_{c_i} \otimes V_{b_i - 1} \] and, \[ \A(\pi_2) = \sum_{i=1}^{r} \phi_i \otimes V_{c_i} \otimes V_{a_i-1} \oplus \sum_{i=r+1}^{r+s} \phi_i \otimes V_{c_i} \otimes V_{b_i}. \]

The decomposition of a relevant pair of representations into components in such a manner as above is unique. We refer to the discussion before Lemma 3.1 in \cite{ggp20} for the proof. Thus given a relevant pair $(\pi_1,\pi_2)$ corresponding to each irreducible term of the form $\phi \otimes V_{a} \otimes V_{b}$ in the decomposition $\A(\pi_1)$, there exists a unique term of the form $\phi \otimes V_{a} \otimes V_{b+1}$ or $\phi \otimes V_{a} \otimes V_{b-1}$ in $\A(\pi_2)$ matching to it. However as the next example shows we do not have the uniqueness of matching of factors in Definition \ref{definition of strong ext relevance} of strong Ext relevance. 

Consider the representations $\pi_1$ and $\pi_2$ of $\G_{13}$ and $\G_{12}$ respectively whose A-parameters as given as follows. \[ \A(\pi_1) = \mathbbm{1} \otimes V_{1} \otimes V_{7} \oplus \mathbbm{1} \otimes V_{5} \otimes V_{1} \oplus \phi \otimes V_{1} \otimes V_{1}, \]   \[ \A(\pi_2) = \mathbbm{1} \otimes V_{1} \otimes V_{6} \oplus \mathbbm{1} \otimes V_{6} \otimes V_{1},\] where, $\phi$ corresponds to a ramified character $\chi$ of $\G_1$. Note that $\pi_1$ and $\pi_2$ are strong Ext relevant. Then one may match the terms in the above decomposition in two different ways.
\begin{enumerate}
    \item The term $\mathbbm{1} \otimes V_{1} \otimes V_{7}$ matches with the term $\mathbbm{1} \otimes V_{1} \otimes V_{6}$ by decreasing the dimension of the representation of the Arthur $\SL_2$ by $1$. The term $\mathbbm{1} \otimes V_{5} \otimes V_{1}$ becomes zero when the dimension of the representation of the Arthur $\SL_2$ is decreased by $1$. Similarly the term $\mathbbm{1} \otimes V_{6} \otimes V_{1}$ becomes zero when the dimension of the representation of the Arthur $\SL_2$ is decreased by $1$.
    \item The term $\mathbbm{1} \otimes V_{1} \otimes V_{7}$ matches with the term $\mathbbm{1} \otimes V_{6} \otimes V_{1}$ by decreasing the the dimension of the representation of the Arthur $\SL_2$ by $1$ and then applying the Aubert-Zelevinsky duality. The term $\mathbbm{1} \otimes V_{5} \otimes V_{1}$ matches with the term $\mathbbm{1} \otimes V_{1} \otimes V_{6}$ (we first decrease the dimension of the representation of the Arthur $\SL_2$ of $\mathbbm{1} \otimes V_{1} \otimes V_{6}$ by $1$ and then apply the Aubert-Zelevinsky duality to obtain the term $\mathbbm{1} \otimes V_{5} \otimes V_{1}$).
\end{enumerate}
In both cases above the term $\phi \otimes V_{1} \otimes V_{1}$ becomes zero when the dimension of the representation of the Arthur $\SL_2$ is decreased by $1$. The above example shows that the decomposition of a strong Ext relevant pair $(\pi_1,\pi_2)$ as in Definition \ref{definition of strong ext relevance} is not unique.

We may explain this example in representation theoretic terms as follows. Recall that given an irreducible representation $\pi$ of $\GL_n(F)$, we we set $\pi^{-}= \nu^{1/2}\pi^{(h)}$ where, $\pi^{(h)}$ denotes the highest derivative of $\pi$. 

Let $\rho$ denote the trivial representation of $\G_1$. In the example above we have $\pi_1 = u_{\rho}(1,7) \times u_{\rho}(5,1) \times \chi$ and $\pi_2 = u_{\rho}(1,6) \times u_{\rho}(6,1)$. One may decompose $\pi_1$ and $\pi_2$ as, \begin{align*}  \pi_1 &= u_{\rho}(1,7) \times u_{\rho}(5,1) \times \chi \times u_{\rho}(6,1)^-  \\ \pi_2 &= u_{\rho}(1,7)^- \times u_{\rho}(5,1)^- \times \chi^- \times u_{\rho}(6,1). \end{align*} This corresponds to the decomposition in the first case above. 

On the other hand we may decompose $\pi_1$ and $\pi_2$ as, \begin{align*} \pi_1 &= u_{\rho}(1,7) \times D(u_{\rho}(1,6)^-) \times \chi \\ \pi_2 &= D(u_{\rho}(1,7)^-) \times u_{\rho}(1,6) \times \chi^-. \end{align*} This corresponds to the decomposition in the second case above.

We remark that a computation of the Ext branching for the above example is carried out in \cite[Example 7.6]{ch22}. We have that, \[ \Ext^i_{\G_{12}}(\pi_1,\pi_2) = \Ext^i_{\G_{11}}(\nu^{1/2}\pi_1^{(2)},\pi_2^{(1)}) \oplus \Ext^i_{\G_{6}}(\nu^{1/2}\pi_1^{(7)},\pi_2^{(6)}). \] This is an instance of a general conjecture about Ext branching for Arthur type representations made in \cite[Conjecture 7.1]{ch22}.

\section*{Acknowledgements} The author would like to thank Professor Dipendra Prasad for his comments on the paper and for helpful discussions especially on Definition \ref{definition of strong ext relevance} of strong $\Ext$ relevance. The author would also like to thank Professor K.Y. Chan for helpful communications through Professor Dipendra Prasad. The author thanks the referees for their suggestions which have improved the paper. This work was supported by an Institute Fellowship of the Indian Institute of Technology, Bombay.


\end{document}